\documentclass[11pt,a4paper]{article}

\usepackage[T1]{fontenc}
\usepackage[utf8]{inputenc}
\usepackage{amsmath,amssymb,amsthm,mathtools}
\usepackage{newtxtext,newtxmath}
\usepackage[margin=25mm,headheight=14pt,headsep=7mm,footskip=11mm]{geometry}
\usepackage{microtype}
\usepackage{graphicx}
\usepackage{flafter}
\usepackage{needspace}
\usepackage{titling}
\usepackage{booktabs,tabularx,array}
\usepackage[font=small,labelfont=bf,labelsep=period]{caption}
\usepackage{subcaption}
\usepackage{authblk}
\makeatletter
\renewcommand{\AB@authnote}[1]{\protect\textsuperscript{#1}}
\renewcommand{\AB@affilnote}[1]{\protect\textsuperscript{#1}}
\makeatother
\usepackage{fancyhdr}
\usepackage{aliascnt}
\usepackage{cite}
\usepackage[hidelinks,bookmarksopen=true,bookmarksnumbered=true]{hyperref}
\usepackage[nameinlink,noabbrev]{cleveref}

\newtheorem{theorem}{Theorem}[section]
\newaliascnt{lemma}{theorem}
\newtheorem{lemma}[lemma]{Lemma}
\aliascntresetthe{lemma}
\newaliascnt{proposition}{theorem}
\newtheorem{proposition}[proposition]{Proposition}
\aliascntresetthe{proposition}
\newaliascnt{corollary}{theorem}
\newtheorem{corollary}[corollary]{Corollary}
\aliascntresetthe{corollary}
\theoremstyle{remark}
\newaliascnt{remark}{theorem}

\aliascntresetthe{remark}
\crefname{theorem}{Theorem}{Theorems}
\crefname{lemma}{Lemma}{Lemmas}
\crefname{proposition}{Proposition}{Propositions}
\crefname{corollary}{Corollary}{Corollaries}
\crefname{remark}{Remark}{Remarks}
\crefname{section}{Section}{Sections}
\crefname{subsection}{Section}{Sections}
\crefname{figure}{Figure}{Figures}
\crefname{table}{Table}{Tables}
\numberwithin{equation}{section}

\newcommand{\C}{\mathbb C}
\newcommand{\R}{\mathbb R}
\newcommand{\Cp}{\mathbb C_+}
\newcommand{\Cn}{\mathcal C_n}
\newcommand{\Un}{\mathsf U_n}
\newcommand{\AL}{A_{L,n}}
\newcommand{\ii}{\mathrm i}
\newcommand{\ee}{\mathrm e}
\DeclareMathOperator{\Arg}{Arg}
\DeclareMathOperator{\cl}{cl}
\DeclareMathOperator{\interior}{int}
\DeclareMathOperator{\Real}{Re}
\DeclareMathOperator{\Imag}{Im}

\title{\bfseries\boldmath On the Spectral Region of $n$-Cycle\protect\\ Stochastic Matrices}
\author[1,2]{Brecht Verbeken}
\author[1,2,3]{Vincent Ginis}
\affil[1]{Department of Business Technology and Operations, Data Analytics Laboratory\\
Vrije Universiteit Brussel (VUB), Pleinlaan 2, 1050 Brussels, Belgium}
\affil[2]{imec-SMIT, Vrije Universiteit Brussel, Pleinlaan 9, 1050 Brussels, Belgium}
\affil[3]{School of Engineering and Applied Sciences, Harvard University\\
Cambridge, Massachusetts 02138, USA}
\affil[ ]{Corresponding author: \href{mailto:brecht.verbeken@vub.be}{\texttt{brecht.verbeken@vub.be}}}
\date{13 September 2026}
\hypersetup{
  pdftitle={On the Spectral Region of n-Cycle Stochastic Matrices},
  pdfauthor={Brecht Verbeken and Vincent Ginis},
  pdfsubject={Exact single-eigenvalue region for stochastic matrices on a directed cycle with self-loops},
  pdfkeywords={stochastic matrices, prescribed zero pattern, spectral region, convexity, algebraic boundary, 15A18, 15B51}
}

\begin{document}
\maketitle
\thispagestyle{plain}

\begin{abstract}
For every $n\geq3$, we determine the union of the spectra of row-stochastic matrices supported on a directed $n$-cycle with self-loops and strictly positive cycle edges. These matrices describe progression through cyclic stages with geometric waiting times, and also arise by uniformization of unidirectional continuous-time reaction cycles. For the latter, the classification determines the minimum transition-rate bound required to realize a prescribed nonreal eigenvalue and provides rates attaining that bound. The angles $m=\Arg\lambda$ and $M=\Arg(\lambda-1)$ reduce the nonreal eigenvalue problem to determining the range of a strictly convex logarithmic sum over angles with prescribed sum. Jensen's inequality gives its minimum, simplex vertices give its finite maximum, and connectedness makes these bounds sufficient for realization. The resulting branch graphs are strictly decreasing and their horizontal sections are ordered. This ordering gives the complete boundary, which alternates between uniform-cycle segments and one-loop algebraic arcs. We state a direct membership criterion and construct a realizing matrix for every spectral point; every nonreal point admits a realization with at most two distinct self-loop weights. The real spectral set is $[-1,1]$ for even $n$ and $(0,1]$ for odd $n$. Allowing deleted cycle edges adds only $0$ in odd dimension and adds nothing in even dimension. We also apply the classification to a branching network with matched holding probabilities on parallel paths. The proofs are independent of Karpelevich's theorem.
\end{abstract}


\section{The problem and the main result}\label{sec:introduction}

For $n\geq2$, let $A_n(\alpha)$ be the matrix whose only possibly nonzero entries are
\begin{equation}\label{eq:family}
 (A_n(\alpha))_{j,j}=\alpha_j,
 \qquad (A_n(\alpha))_{j,j+1}=1-\alpha_j,
 \qquad 0\leq\alpha_j<1,
\end{equation}
where indices are taken modulo $n$. Thus every cycle edge has positive weight, while self-loops may be absent. We study
\begin{equation}\label{eq:region-definition}
 \Cn=\{A_n(\alpha):\alpha\in[0,1)^n\},
 \qquad \Sigma_n=\bigcup_{A\in\Cn}\sigma(A).
\end{equation}
The half-open parameter interval is part of the problem: replacing it by $[0,1]$ would allow cycle edges to disappear. In particular, boundary points obtained by taking parameter limits must be distinguished from eigenvalues of admissible matrices.

A concrete question arises from the continuous-time version of this cycle. At stage $j$, the chain advances to stage $j+1$ at rate $r_j>0$, so its generator has entries $Q_{j,j}=-r_j$ and $Q_{j,j+1}=r_j$. Suppose that the number of stages is fixed and that one wishes to realize an eigenmode with decay rate $d>0$ and angular frequency $\omega>0$. What is the smallest upper bound on the rates that permits the eigenvalue $\nu=-d+\ii\omega$, and which rates attain it? The time factor $\ee^{t\nu}$ has decay time $1/d$ and oscillation period $2\pi/\omega$. Uniformization turns this question into membership in $\Sigma_n$. In \cref{cor:minimum-rate}, the complete region gives the exact minimum, attained with all but at most one transition at the upper bound. For three stages, the pair $-1\pm\ii/2$ requires a bound of $5/6$, whereas allowing backward transitions reduces the minimum bound on total exit rates to $2/3$. The calculation in \cref{sec:minimum-rates} makes explicit how the prescribed support changes the possible dynamics.

The effect of closing the parameter domain can be stated exactly. For $n\geq3$, let $\widetilde\Sigma_n$ be the spectral union obtained by allowing $\alpha\in[0,1]^n$. Then
\[
 \widetilde\Sigma_n=\Sigma_n\cup\{0\}=\cl\Sigma_n.
\]
Thus deleting a cycle edge adds no nonreal points: it adds only $0$ when $n$ is odd and adds nothing when $n$ is even. The determinant argument and the endpoint classification proving this statement are given in \cref{cor:closed-domain}. The distinction between an admissible parameter and its limit remains necessary when constructing a matrix for a prescribed boundary point.

The unrestricted stochastic single-eigenvalue problem was settled by Dmitriev and Dynkin and by Karpelevich~\cite{DmitrievDynkin,Karpelevich}. Subsequent work has clarified polynomial descriptions of its boundary, matrix realizations, and the structure of its proof~\cite{Ito,JohnsonPaparella,KimKim,KirklandSmigoc,MungerNickersonPaparella}. A prescribed support introduces a different question: the ambient stochastic region does not determine which points can still be realized after entries are forced to vanish. Related restrictions lead to the doubly stochastic problem~\cite{PerfectMirsky,MashreghiRivard,LevickPereiraKribs,HarlevJohnsonLim} and to eigenvalue regions of monotone stochastic matrices~\cite{VagenendeVerbekenGuerry}.

For prescribed zero patterns, Ran and Teng studied the inverse eigenvalue problem in dimension three and formulated further questions in dimension four~\cite{RanTeng}. The $4$-cycle stochastic region was determined by Vagenende, Verbeken, Algaba, and Guerry~\cite{VagenendeEtAl}. That result is the motivating special case and conceptual precursor of the present classification. Here the dimension is arbitrary, and the boundary is derived directly from the cycle support rather than from the unrestricted stochastic region.

More specifically, the four-cycle paper~\cite{VagenendeEtAl} already passes from the multiplicative characteristic equation to the angles $m,M$, a fixed sum of arguments, and a zero sum of logarithmic moduli. Its convexity argument identifies the equal-angle minimum, distinguishes finite and unbounded maxima, and produces the uniform and one-loop realizing families. The present work develops this mechanism for all admissible argument branches in arbitrary dimension. The additional geometric task is to determine which portions of those branches remain visible after their regions are combined. Strictly decreasing branch graphs and an ordering of their horizontal sections resolve that task and give the complete boundary chain. The same ordering yields a construction valid in every dimension. The specialization of the reduced polynomial to the earlier four-cycle equation is recorded in \eqref{eq:four-cycle-specialization}.

The one-loop family also has a direct connection with Leslie matrices. Kirkland determined an eigenvalue region for stochastic Leslie matrices~\cite{KirklandLeslie}. If $\Pi$ has columns $e_1,e_n,e_{n-1},\ldots,e_2$, then $\Pi^{\mathsf T}A_n(\alpha,0,\ldots,0)\Pi$ has top row $(\alpha,0,\ldots,0,1-\alpha)$ and ones on the subdiagonal. Its characteristic polynomial is
\[
 \lambda^n-\alpha\lambda^{n-1}+\alpha-1.
\]
Thus the one-loop realization is a particular stochastic Leslie matrix after relabeling. The full family studied here allows holding at every stage, and the question is which parts of these one-loop branches bound its entire spectral union. Kirkland's work on near-periodic fecundity patterns~\cite{KirklandNearPeriodic} further illustrates the dynamical role of nonreal eigenvalues near the spectral radius: they govern oscillatory modes in age-structured population models.

The distinction between a polynomial root and a point on the boundary also occurs in the unrestricted stochastic problem. Kirkland, Laffey, and \v{S}migoc~\cite{KirklandLaffeySmigoc} identify the boundary root at a prescribed argument when the relevant polynomial has competing candidates. Their subdominant-realization result concerns unrestricted stochastic matrices. Here the classification and reconstruction concern an individual eigenvalue under the cycle support restriction; no assertion that the chosen eigenvalue is subdominant is required.

The rate question connects this support restriction to spectral constraints on biochemical oscillations. Barato and Seifert~\cite{BaratoSeifert2017} formulated conjectural bounds relating oscillation coherence to thermodynamic driving and network topology. For a mode $-d+\ii\omega$, the coherence ratio is $\omega/d$, while $\omega/(2\pi d)$ counts the number of periods in one exponential decay time.

Kolchinsky, Ohga, and Ito~\cite{KolchinskyOhgaIto2026} prove eigenvalue localization bounds involving cycle affinity and eigenvector winding for both unicyclic and more general Markov generators. Their class permits backward transitions and includes irreversible cycles. Their unicyclic proof also combines phase increments with Jensen's inequality for a convex logarithmic expression~\cite[Eqs.~(12)--(13)]{KolchinskyOhgaIto2026}; the additional task here is to determine the complete attainable range and reconstruct admissible parameters for the one-way support. For the strictly one-way support considered here, \cref{cor:minimum-rate} gives an exact answer to the rate question: when the target is feasible, it determines the smallest possible maximum stage rate and provides rates attaining it. The complete three-stage calculation in \cref{sec:minimum-rates} also shows explicitly how allowing backward transitions changes this minimum. These statements concern realization of a prescribed individual mode; they impose no condition that it be the slowest decaying mode of the full chain.

The cases $n=1,2$ require no geometric analysis. For $n=1$ the cycle edge and self-loop coincide, so the only matrix is $[1]$. For $n=2$ the eigenvalues are $1$ and $\alpha_1+\alpha_2-1$. Hence
\[
 \Sigma_1=\{1\},\qquad \Sigma_2=[-1,1].
\]
From now on, $n\geq3$ and $\Cp=\{z\in\C:\Imag z>0\}$.

\subsection{Cyclic stages and their spectral modes}\label{sec:stage-motivation}

The support in \eqref{eq:family} has an immediate probabilistic interpretation. At stage $j$, a discrete-time Markov chain either waits at that stage, with probability $\alpha_j$, or advances to stage $j+1$, with probability $q_j=1-\alpha_j$. The number $G_j$ of steps spent at the stage before advancing has distribution
\[
 \mathbb P(G_j=t)=q_j\alpha_j^{t-1}\quad(t=1,2,\ldots),
 \qquad \mathbb E G_j=\frac1{q_j}.
\]
Successive stage durations in a completed circuit are independent. The model therefore fixes the order of the stages while allowing their mean durations to vary independently. Its stationary proportions also retain that interpretation. If $\pi$ is stationary, balance at stage $j$ gives $\pi_jq_j=\pi_{j-1}q_{j-1}$, so normalization yields
\[
 \pi_j=\frac{q_j^{-1}}{\sum_{\ell=1}^n q_\ell^{-1}}.
\]
Stages with longer mean durations have larger stationary proportions.

The eigenvalues describe how observables change as the chain progresses. For a column vector $f$ with $A_n(\alpha)f=\lambda f$, and a chain $(X_t)_{t\geq0}$ started at state $i$,
\[
 \mathbb E_i f(X_t)=(A_n(\alpha)^t f)_i=\lambda^t f_i.
\]
When $\lambda=\rho\ee^{\ii\vartheta}$ is nonreal, the real and imaginary parts of $f$ span a real invariant space on which the evolution combines the factor $\rho^t$ with angular increment $\vartheta$ per step. The question studied here is which such factors can occur when the number of cyclic stages is prescribed and the holding probabilities may vary. Membership in $\Sigma_n$ answers existence, while the realizing matrices supply the stage probabilities. The question leaves other eigenvalues and the relative size of the chosen mode unrestricted.

A continuous-time version occurs in unidirectional linear reaction cycles. The one-molecule model of Mukherjee, Ghosh, and De~\cite{MukherjeeGhoshDe} has successive transformations $X_j\to X_{j+1}$ at positive rates $r_j$, with cyclic indices. Its generator $Q$ has entries $Q_{j,j}=-r_j$ and $Q_{j,j+1}=r_j$. For any $c\geq\max_j r_j$, uniformization gives
\[
 P=I+Q/c=A_n(1-r_1/c,\ldots,1-r_n/c).
\]
A Poisson clock of rate $c$ together with transition matrix $P$ describes the same continuous-time chain, and its eigenvalues correspond by $\mu=c(\lambda-1)$. This connects the classification to the possible modes of a reaction cycle with bounded stage rates. The matrix $P$ is the uniformized transition matrix; observing the chain at a fixed time interval instead gives $\exp(tQ)$, which generally permits transitions across several stages. The precise generator statement is proved in \cref{cor:generator-region}.

The same stage interpretation explains the product structure of the proof. Completing a circuit requires traversing every stage once, and the associated generating functions multiply. \Cref{prop:renewal} makes this relation exact and distinguishes a completed circuit from an ordinary return to a state, which may consist of a single self-loop. In angular coordinates the product becomes a fixed argument sum and a logarithmic modulus sum. The range of the latter, rather than the convexity of its zero set, supplies the existence criterion. The next subsection introduces the coordinates used to read that range geometrically.

\subsection{Coordinates adapted to the two distinguished points}

For $\lambda\in\Cp$, put
\begin{equation}\label{eq:angular-coordinates}
 m=\Arg\lambda,\qquad M=\Arg(\lambda-1),
 \qquad D=\{(m,M):0<m<M<\pi\}.
\end{equation}
The two angles record the directions of $\lambda$ from $0$ and from $1$. Their inverse map is
\begin{equation}\label{eq:Lambda}
 \Lambda(m,M)=\frac{\sin M}{\sin(M-m)}\ee^{\ii m}.
\end{equation}
In particular, fixing $M$ means moving along a ray starting at $1$. This makes horizontal sections of the angular region useful for identifying the global boundary.

Let
\begin{equation}\label{eq:branch-angles}
 \begin{gathered}
 K=\left\lfloor\frac{n-1}{2}\right\rfloor,\qquad
 \theta_k=\frac{2\pi k}{n},\qquad
 \phi_k=\frac{\pi+\theta_k}{2},\qquad
 \gamma_k=\frac{2\pi k-\pi}{n-1}
 \quad (1\leq k\leq K).
 \end{gathered}
\end{equation}
The angle $\theta_k$ locates the root of unity $\omega_k=\ee^{\ii\theta_k}$. The level $\phi_k$ is the direction from $1$ to $\omega_k$, and $\gamma_k$ will be the limiting direction of an algebraic arc at $0$.

For $m\in[\gamma_k,\theta_k]$, consider the scalar equation
\begin{equation}\label{eq:radial-polynomial}
 \rho^n\sin m+\rho\sin(2\pi k-nm)
       =\sin(2\pi k-(n-1)m).
\end{equation}
The theorem below includes the existence and uniqueness needed to use this equation as a definition. Write its solution in $[0,1]$ as $\rho_k(m)$, and set
\begin{equation}\label{eq:beta-definition}
 \beta_k(m)=\Arg\bigl(\rho_k(m)\ee^{\ii m}-1\bigr),
 \qquad \beta_k(\gamma_k)=\pi.
\end{equation}
Once strict monotonicity is established, write
\[
 h_k=\beta_k^{-1}:[\phi_k,\pi)\longrightarrow(\gamma_k,\theta_k].
\]

We call branch $k$ active at height $M$ when $\phi_k\leq M$.

\Needspace{18\baselineskip}
\begin{theorem}[Spectral region]\label{thm:region}
For each $1\leq k\leq K$, \eqref{eq:radial-polynomial} has a unique solution $\rho_k(m)\in[0,1]$ for every $m\in[\gamma_k,\theta_k]$. The function $\beta_k$ is continuous and strictly decreasing, with
\[
 \rho_k(\gamma_k)=0,\qquad \rho_k(\theta_k)=1,
 \qquad \beta_k(\gamma_k)=\pi,\qquad \beta_k(\theta_k)=\phi_k.
\]
For $\phi_1\leq M<\pi$, define the largest active index and the angular endpoint function by
\begin{equation}\label{eq:endpoint-function}
 \kappa(M)=\max\{k\in\{1,\ldots,K\}:\phi_k\leq M\},
 \qquad \mu_n(M)=h_{\kappa(M)}(M).
\end{equation}
No upper-half-plane spectral point has $M<\phi_1$. For $(m,M)\in D$ with $M\geq\phi_1$, the exact membership criterion is
\begin{equation}\label{eq:main-membership}
 \Lambda(m,M)\in\Sigma_n\quad\Longleftrightarrow\quad m\leq\mu_n(M).
\end{equation}
The lower-half-plane region is its complex conjugate, and
\begin{equation}\label{eq:real-set}
 \Sigma_n\cap\R=
 \begin{cases}
 [-1,1],&n\text{ even},\\
 (0,1],&n\text{ odd},
 \end{cases}
 \qquad (0,1)\subset\interior\Sigma_n.
\end{equation}
The set $\Sigma_n$ is star-convex with centre $1$. The point $0$ belongs to its boundary for every $n\geq3$, and belongs to $\Sigma_n$ exactly when $n$ is even.
\end{theorem}

\subsection{Deciding membership and constructing a matrix}\label{sec:direct-test}

The endpoint function in \cref{thm:region} describes the geometry, but membership itself has a formulation that requires no inversion of a branch graph. In addition to $\theta_k$ and $\phi_k$, write
\begin{equation}\label{eq:direct-data}
 \begin{gathered}
 U_k(m)=2\pi k-(n-1)m,\\
 H_n(\lambda)=\Imag\bigl((\lambda^n-1)\overline{(\lambda^{n-1}-1)}\bigr).
 \end{gathered}
\end{equation}
The quantity $U_k(m)$ is the greatest value one argument can have under the fixed-sum and lower-bound constraints, before imposing the excluded upper endpoint $M$. It is obtained when the other $n-1$ coordinates equal $m$. The polynomial expression $H_n$ records the sign of the finite upper end of the logarithmic range, under the conditions stated below.

\begin{corollary}[Direct membership criterion]\label{cor:direct-membership}
Let $n\geq3$ and $\lambda\in\Cp$, with $m=\Arg\lambda$ and $M=\Arg(\lambda-1)$. If $M<\phi_1$, then $\lambda\notin\Sigma_n$. For $\phi_1\leq M<\pi$, set
\begin{equation}\label{eq:active-floor}
 k=\kappa(M)=\left\lfloor n\left(\frac{M}{\pi}-\frac12\right)\right\rfloor.
\end{equation}
Then $1\leq k\leq K$ and
\begin{equation}\label{eq:direct-membership}
 \lambda\in\Sigma_n
 \quad\Longleftrightarrow\quad
 m\leq\theta_k\ \text{ and }\
 \bigl(M\leq U_k(m)\ \text{or}\ H_n(\lambda)\geq0\bigr).
\end{equation}
For a lower-half-plane point, apply the criterion to its conjugate. Real points are governed by \eqref{eq:real-set}.
\end{corollary}

The proof is given after \cref{prop:carrier} in \cref{sec:algebra}. The largest active branch suffices because its horizontal section contains all earlier active sections. On that branch the condition $M\leq U_k(m)$ needs no algebraic sign restriction, including at equality. The alternative $M>U_k(m)$ is precisely the regime in which the denominator used to compare $H_n$ with the finite maximum is positive. Thus the angular conditions remain part of the criterion. The formula \eqref{eq:active-floor} is used only for $M<\pi$; the real axis is treated separately.

There is also a short structural description of the realizing matrices. \Cref{cor:reconstruction} proves that every nonreal $\lambda\in\Sigma_n$ is an eigenvalue of a matrix
\[
 A_n(a,b,\ldots,b),\qquad 0\leq b\leq a<1.
\]
Consequently, all but one stage can be given the same holding probability. The full construction remains in \cref{sec:real}, where the geometric meaning of its parameters can be seen directly. Membership, boundary parametrization, and realization are distinct uses of the classification, and the formulas below address each of them.

Between consecutive levels $\phi_k$, the endpoint $\mu_n(M)$ follows the inverse of one algebraic arc. At a new level $\phi_{k+1}$, it jumps to the next root of unity. The jump supplies a straight boundary segment. Thus the upper boundary is an alternating chain from $1$ to $0$; \cref{thm:boundary} gives all endpoints and parameter intervals. \Cref{fig:odd-even} compares the odd- and even-dimensional regions.

\begin{figure}[!htbp]
 \centering
 \begin{subfigure}[t]{0.485\textwidth}
   \centering
   \includegraphics[width=\linewidth]{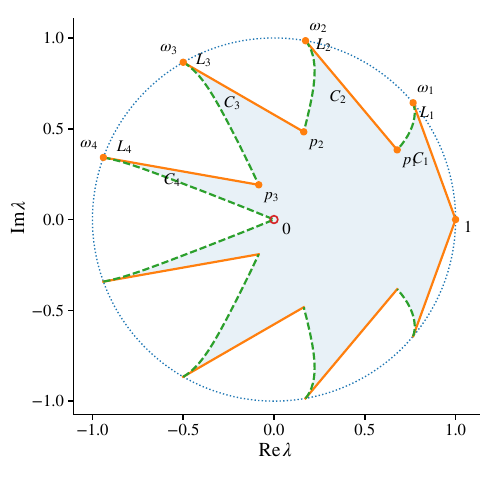}
   \caption{Odd dimension: $n=9$.}\label{fig:odd}
 \end{subfigure}\hfill
 \begin{subfigure}[t]{0.485\textwidth}
   \centering
   \includegraphics[width=\linewidth]{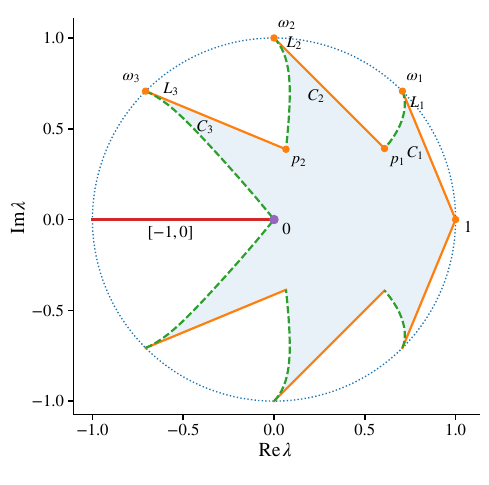}
   \caption{Even dimension: $n=8$.}\label{fig:even}
 \end{subfigure}
 \caption{Spectral regions in the complex plane. The upper boundary alternates between line segments $C_k$ realized by uniform matrices (solid) and one-loop arcs $L_k$ (dashed), meeting at $\omega_k=\ee^{2\pi\ii k/n}$ and $p_k$. The lower boundary is the conjugate chain; the dotted circle is $|\lambda|=1$. The open marker at $0$ in the odd case denotes a boundary point that is not attained. In the even case, $0$ is attained and the additional real boundary interval $[-1,0]$ is shown separately. The shaded parts depict the two-dimensional regions.}
 \label{fig:odd-even}
\end{figure}

\Needspace{7\baselineskip}
\subsection{Why convexity produces this geometry}

Only two permutations contribute to the characteristic determinant. After passing to arguments, the eigenvalue equation therefore imposes a fixed sum of angles and a zero sum of logarithmic moduli. The latter sum is strictly convex. Its minimum occurs when all angles agree; its finite maximum occurs when all excess angle is placed in one coordinate. These two configurations become, respectively, uniform matrices and matrices with at most one nonzero self-loop.

The proof follows this mechanism. \Cref{sec:feasibility} makes the angular reduction exact, including the inverse map to admissible parameters. \Cref{sec:range} finds the complete range of the convex sum. \Cref{sec:sectors} turns the resulting inequalities into sectors and identifies their realizing families. \Cref{sec:boundary} orders their horizontal sections and extracts the boundary. \Cref{sec:real} treats the real axis and gives a constructive radial filling, followed by a realization summary in \cref{tab:realization}. \Cref{sec:algebra} records the polynomial equation, its sign test, and continuation of the one-loop eigenvalue branches. None of these algebraic refinements is needed to justify the boundary extraction.

\Cref{sec:models} then returns to the stochastic interpretation. It proves the circuit generating-function identity and the continuous-time correspondence, and applies the classification to two parallel return paths with matching holding probabilities. The latter example also makes explicit which factorization survives when the support branches and which part of the convex angular argument requires additional structure.

\section{From the characteristic equation to angular feasibility}\label{sec:feasibility}

We first record three elementary properties. Since $A_n(\alpha)\mathbf1=\mathbf1$ and $\|A_n(\alpha)\|_\infty=1$, every matrix has eigenvalue $1$ and every eigenvalue satisfies $|\lambda|\leq1$. Real entries give conjugation symmetry. Finally, for $0<s\leq1$,
\begin{equation}\label{eq:affine-realization}
 (1-s)I+sA_n(\alpha)
 =A_n\bigl(1-s(1-\alpha_1),\ldots,1-s(1-\alpha_n)\bigr).
\end{equation}
Every parameter on the right belongs to $[0,1)$. An eigenvalue $\lambda$ is sent to $1+s(\lambda-1)$. Together with the always-present eigenvalue $1$, this proves star-convexity. At $s=0$ the matrix on the left is $I$, which is not in the specified family; the spectral endpoint $1$ is attained separately, so no closed-parameter extension is being used.

\begin{lemma}\label{lem:product}
For every $\lambda\in\C$ and $\alpha\in[0,1)^n$,
\begin{equation}\label{eq:determinant}
 \det(\lambda I-A_n(\alpha))
   =\prod_{j=1}^n(\lambda-\alpha_j)-\prod_{j=1}^n(1-\alpha_j).
\end{equation}
Consequently, $\lambda$ is an eigenvalue exactly when
\begin{equation}\label{eq:normalized-product}
 \prod_{j=1}^n\frac{\lambda-\alpha_j}{1-\alpha_j}=1.
\end{equation}
\end{lemma}

\begin{proof}
A nonzero permutation term chooses in each row either the diagonal or the next column. Choosing the next column in any row forces the same choice in the following row, and hence all the way around the cycle. The only contributing terms are therefore the diagonal term and the full cycle. The cycle has permutation sign $(-1)^{n-1}$ and entry sign $(-1)^n$, so its contribution is negative. This proves \eqref{eq:determinant}. Division by $\prod_j(1-\alpha_j)>0$ gives \eqref{eq:normalized-product}.
\end{proof}

\begin{lemma}\label{lem:angles}
The map $\Lambda:D\to\Cp$ in \eqref{eq:Lambda} is a homeomorphism, with inverse \eqref{eq:angular-coordinates}. Fix $\lambda=a+\ii b\in\Cp$ and its coordinates $(m,M)$. The map
\[
 \alpha\longmapsto u=\Arg(\lambda-\alpha)
\]
is a continuous strictly increasing bijection from $[0,1)$ to $[m,M)$, with inverse
\begin{equation}\label{eq:inverse-parameter}
 \alpha(u)=a-b\cot u.
\end{equation}
Moreover,
\begin{equation}\label{eq:modulus-function}
 \log\frac{|\lambda-\alpha(u)|}{1-\alpha(u)}
   =f_M(u):=\log\frac{\sin M}{\sin(M-u)},
 \qquad m\leq u<M.
\end{equation}
The function $f_M$ is strictly convex and diverges at the excluded endpoint:
\begin{equation}\label{eq:convexity}
 f'_M(u)=\cot(M-u),\qquad
 f''_M(u)=\csc^2(M-u)>0,\qquad
 \lim_{u\uparrow M}f_M(u)=+\infty.
\end{equation}
\end{lemma}

\begin{proof}
For $\lambda=a+\ii b$ with $b>0$,
\[
 \cot m=\frac{a}{b},\qquad \cot M=\frac{a-1}{b}.
\]
Thus $0<m<M<\pi$ and
\[
 b=\frac{1}{\cot m-\cot M}
   =\frac{\sin m\sin M}{\sin(M-m)},\qquad a=b\cot m.
\]
Conversely, these expressions define $a\in\R$ and $b>0$ for every $(m,M)\in D$, with the required arguments. They prove \eqref{eq:Lambda} and continuity in both directions, and give the companion identity
\begin{equation}\label{eq:ray-identity}
 \Lambda(m,M)-1=\frac{\sin m}{\sin(M-m)}\ee^{\ii M}.
\end{equation}

For the parameter map, $\cot u=(a-\alpha)/b$ gives \eqref{eq:inverse-parameter}. The derivative of this inverse is $b\csc^2u>0$, and its endpoints are $\alpha(m)=0$ and $\lim_{u\uparrow M}\alpha(u)=1$. Furthermore,
\[
 |\lambda-\alpha(u)|=\frac{b}{\sin u},\qquad
 1-\alpha(u)=b(\cot u-\cot M).
\]
Using
\[
 \cot u-\cot M=\frac{\sin(M-u)}{\sin u\sin M}
\]
cancels both $b$ and $\sin u$, giving \eqref{eq:modulus-function}. Differentiation proves \eqref{eq:convexity}.
\end{proof}

The cancellation in \eqref{eq:modulus-function} is the reason for these coordinates. For fixed $M$, the modulus function has no further dependence on $\lambda$; the other angle $m$ enters only through the allowed interval $[m,M)$.

For $1\leq k\leq K$, define
\begin{equation}\label{eq:feasible-set}
 P_k(m,M)=\left\{u\in[m,M)^n:\sum_{j=1}^n u_j=2\pi k\right\},
 \qquad \Psi_M(u)=\sum_{j=1}^n f_M(u_j).
\end{equation}
We call $k$ the \emph{argument branch}. It records the integer multiple of $2\pi$ accumulated by the arguments of the factors in \eqref{eq:normalized-product}, not a choice of logarithm for an individual factor.

The argument branch also describes the phase of an eigenvector around the cycle. Let $f\ne0$ satisfy $A_n(\alpha)f=\lambda f$ with $\lambda\in\Cp$. The row equations are
\[
 (1-\alpha_j)f_{j+1}=(\lambda-\alpha_j)f_j.
\]
If one component vanished, these equations would force every component to vanish. Thus all ratios below are defined, and
\[
 \begin{gathered}
 \frac{f_{j+1}}{f_j}=\frac{\lambda-\alpha_j}{1-\alpha_j},
 \qquad \Arg\frac{f_{j+1}}{f_j}=u_j\in(0,\pi),\\
 \frac1{2\pi}\sum_{j=1}^n\Arg\frac{f_{j+1}}{f_j}=k.
 \end{gathered}
\]
The last identity says that the eigenvector accumulates $k$ full turns of phase as the stages are traversed. Since $1\leq k\leq K<n/2$, it agrees with the wrapped winding number in \cite[Eq.~(2)]{KolchinskyOhgaIto2026}. Uniformization preserves the eigenvector, so the same interpretation holds for the corresponding generator mode. Here $k$ belongs to a fixed matrix and eigenpair. By contrast, $\kappa(M)$ in \cref{cor:direct-membership} selects a branch sufficient to decide membership in the union. This selection does not assert that every matrix realizing the same $\lambda$ has winding number $\kappa(M)$.

\begin{proposition}\label{prop:feasibility}
For $\lambda=\Lambda(m,M)\in\Cp$,
\begin{equation}\label{eq:exact-feasibility}
 \lambda\in\Sigma_n
 \quad\Longleftrightarrow\quad
 \text{there are }1\leq k\leq K\text{ and }u\in P_k(m,M)
 \text{ with }\Psi_M(u)=0.
\end{equation}
Every feasible vector gives an admissible realizing matrix by \eqref{eq:inverse-parameter}.
\end{proposition}

\begin{proof}
For the angle corresponding to $\alpha_j$, \eqref{eq:modulus-function} yields
\[
 \frac{\lambda-\alpha_j}{1-\alpha_j}
   =\exp\bigl(f_M(u_j)+\ii u_j\bigr).
\]
Their product is $1$ exactly when $\sum_j f_M(u_j)=0$ and $\sum_j u_j$ is an integer multiple of $2\pi$. Each $u_j\in(0,\pi)$, so the admissible multiples are precisely $2\pi k$ with $1\leq k\leq K$. Conversely, the bijection in \cref{lem:angles} sends every $u_j\in[m,M)$ to $\alpha_j\in[0,1)$. The same identity then recovers \eqref{eq:normalized-product}. Thus the reverse implication constructs an actual matrix, not merely a limit of matrices.
\end{proof}

\section{The exact range of the convex sum}\label{sec:range}

Fix an argument branch and abbreviate $\theta=\theta_k$. Its feasible set is nonempty precisely when
\begin{equation}\label{eq:nonempty}
 m\leq\theta<M.
\end{equation}
Necessity follows by averaging the coordinates, and sufficiency follows from the equal-angle vector $(\theta,\ldots,\theta)$. Set
\begin{equation}\label{eq:U}
 U=U_k(m)=2\pi k-(n-1)m.
\end{equation}
This is the largest value one coordinate can have when all other coordinates remain at their lower bound $m$, before imposing the upper restriction $u_j<M$.

\begin{proposition}[Range on an argument branch]\label{prop:range}
Assume \eqref{eq:nonempty}. The minimum is
\begin{equation}\label{eq:minimum}
 \min_{u\in P_k(m,M)}\Psi_M(u)=nf_M(\theta_k),
\end{equation}
and is attained only at the equal-angle vector. If $U<M$, then
\begin{equation}\label{eq:maximum}
 \max_{u\in P_k(m,M)}\Psi_M(u)=(n-1)f_M(m)+f_M(U).
\end{equation}
The maximum is attained precisely at the permutations of $(U,m,\ldots,m)$; when $U=m$, these coincide with the unique feasible vector. If $U\geq M$, the supremum is $+\infty$. More precisely,
\begin{equation}\label{eq:complete-range}
 \Psi_M(P_k(m,M))=
 \begin{cases}
 [nf_M(\theta_k),(n-1)f_M(m)+f_M(U)],&U<M,\\
 [nf_M(\theta_k),+\infty),&U\geq M.
 \end{cases}
\end{equation}
\end{proposition}

\begin{proof}
Jensen's inequality and strict convexity give
\[
 \frac1n\Psi_M(u)\geq f_M\left(\frac1n\sum_j u_j\right)=f_M(\theta_k),
\]
with equality exactly when all coordinates agree. The equal-angle vector is feasible, so the lower endpoint is attained.

Suppose first that $U<M$. Write $u_j=m+y_j$. The constraints become
\[
 y_j\geq0,\qquad \sum_j y_j=U-m.
\]
Since each $y_j\leq U-m<M-m$, the excluded upper endpoint imposes no additional restriction. Thus the feasible set is the compact simplex whose vertices are the permutations of $(U,m,\ldots,m)$. Every point is a convex combination of these vertices. Convexity bounds its value by the corresponding combination of vertex values, and those values are all equal. This proves \eqref{eq:maximum}. Strict convexity excludes a nonvertex maximizer when $U>m$. If $U=m$, the simplex consists of the equal-angle vector and both formulas coincide.

Suppose instead that $U\geq M$. Let $u_n=M-\varepsilon$ and put
\[
 u_1=\cdots=u_{n-1}=\frac{2\pi k-M+\varepsilon}{n-1}.
\]
As $\varepsilon\downarrow0$, the latter coordinates tend to
\[
 v=\frac{2\pi k-M}{n-1}\in[m,M).
\]
Here $v\geq m$ is exactly $U\geq M$, while $v<M$ follows from $2\pi k<nM$. Therefore the vectors are feasible for all sufficiently small positive $\varepsilon$. The last term $f_M(u_n)$ tends to $+\infty$, and the other terms have finite limits.

Finally, $P_k(m,M)$ is convex, hence connected, and $\Psi_M$ is continuous. Its image is an interval. In the compact case both endpoints just computed are attained. In the other case the minimum is attained and the image is unbounded above. This proves \eqref{eq:complete-range}.
\end{proof}

When $U<M$, all excess argument can be concentrated in one coordinate before reaching the forbidden parameter $\alpha=1$. When $U\geq M$, one coordinate can approach that endpoint while the total argument remains fixed, and the logarithmic modulus becomes arbitrarily large. No endpoint at $\alpha=1$ is needed to obtain any finite value of the sum.

In the compact case, write the maximum as
\begin{equation}\label{eq:g}
 \begin{split}
 g_k(m,M)={}&n\log\sin M-(n-1)\log\sin(M-m)\\
           &-\log\sin(M-U_k(m)),\qquad M>U_k(m).
 \end{split}
\end{equation}
All logarithms here have positive arguments. The minimum condition has a particularly simple form:
\begin{equation}\label{eq:jensen-angle}
 f_M(\theta_k)\leq0\quad\Longleftrightarrow\quad M\geq\phi_k.
\end{equation}
Indeed, for $\theta_k<M<\pi$,
\[
 \sin(M-\theta_k)-\sin M=-2\sin(\theta_k/2)\cos(M-\theta_k/2),
\]
which is nonnegative precisely when $M\geq(\pi+\theta_k)/2$.

\begin{corollary}\label{cor:branchwise}
A point $\lambda=\Lambda(m,M)\in\Cp$ belongs to $\Sigma_n$ exactly when at least one $k\in\{1,\ldots,K\}$ satisfies
\begin{equation}\label{eq:branchwise}
 m\leq\theta_k,\qquad M\geq\phi_k,\qquad
 \begin{cases}
 \text{no further condition},&M\leq U_k(m),\\
 g_k(m,M)\geq0,&M>U_k(m).
 \end{cases}
\end{equation}
\end{corollary}

\begin{proof}
By \cref{prop:feasibility,prop:range}, zero must lie in one of the intervals \eqref{eq:complete-range}. Equation~\eqref{eq:jensen-angle} gives its lower-endpoint condition, and \eqref{eq:g} gives its finite upper-endpoint condition. Since $\phi_k>\theta_k$, the requirement $M\geq\phi_k$ already ensures $\theta_k<M$ in \eqref{eq:nonempty}.
\end{proof}

The set $P_k(m,M)$ is convex, but its subset satisfying $\Psi_M=0$ need not be convex. For example, if the finite maximum is zero and $U_k(m)>m$, \cref{prop:range} says that the zero set consists exactly of the distinct vertices obtained by permuting $(U_k(m),m,\ldots,m)$. Strict convexity makes the value at the midpoint of two such vertices strictly negative. What makes the criterion exact is the interval $\Psi_M(P_k(m,M))$, whose endpoints have been determined and whose connectedness supplies every intermediate value.

\section{Angular sectors and the two extremal families}\label{sec:sectors}

We now turn the implicit inequality $g_k\geq0$ into a graph. We establish monotonicity from the logarithmic expression first, and recover the polynomial equation afterwards.

\subsection{The upper edge is a unique decreasing graph}

\begin{proposition}\label{prop:graph}
For $\gamma_k<m\leq\theta_k$, the equation $g_k(m,M)=0$ has a unique solution
\[
 U_k(m)<M<\pi.
\]
Denote it by $\beta_k(m)$. It extends continuously to $[\gamma_k,\theta_k]$, is strictly decreasing there, and satisfies
\begin{equation}\label{eq:beta-endpoints}
 \beta_k(\gamma_k)=\pi,\qquad \beta_k(\theta_k)=\phi_k.
\end{equation}
In particular, $\beta_k(m)\geq\phi_k$. On its open parameter interval it is continuously differentiable, and
\begin{equation}\label{eq:implicit-beta-derivative}
 \beta'_k(m)=-\frac{\partial_mg_k(m,\beta_k(m))}
                         {\partial_Mg_k(m,\beta_k(m))}<0.
\end{equation}
\end{proposition}

\begin{proof}
Fix $\gamma_k<m<\theta_k$ and write $U=U_k(m)$. The definitions give $m<U<\pi$. On $U<M<\pi$,
\begin{align}
 \partial_Mg_k
 &=n\cot M-(n-1)\cot(M-m)-\cot(M-U)<0,\label{eq:g-M}\\
 \partial_mg_k
 &=(n-1)\bigl[\cot(M-m)-\cot(M-U)\bigr]<0.\label{eq:g-m}
\end{align}
For the first sign, both $M-m$ and $M-U$ are smaller than $M$, and cotangent is strictly decreasing on $(0,\pi)$. For the second, $M-U<M-m$. Thus every comparison is within a single interval on which cotangent is monotone.

As $M\downarrow U$, the last term in \eqref{eq:g} tends to $+\infty$ and the others stay finite. As $M\uparrow\pi$, the first term tends to $-\infty$ and the others stay finite. Continuity and \eqref{eq:g-M} therefore give a unique zero. The implicit function theorem and \eqref{eq:g-m} give \eqref{eq:implicit-beta-derivative}.

At $m=\theta_k$, one has $U=m$, so
\[
 g_k(\theta_k,M)=n\log\frac{\sin M}{\sin(M-\theta_k)}.
\]
This is strictly decreasing from $+\infty$ to $-\infty$ on $(\theta_k,\pi)$, with its unique zero at $M=\phi_k$. The same implicit function argument at $(\theta_k,\phi_k)$ gives continuity from the left. At the other endpoint,
\[
 U_k(m)<\beta_k(m)<\pi,\qquad U_k(m)\longrightarrow\pi\quad(m\downarrow\gamma_k),
\]
so $\beta_k(m)\to\pi$. Strict decrease in the interior and these endpoint limits prove strict decrease on the whole closed interval, as well as $\beta_k(m)\geq\phi_k$.
\end{proof}

The graph just constructed will be identified with \eqref{eq:beta-definition} in \cref{prop:one-loop}. Define the $k$th sector by its vertical sections:
\begin{equation}\label{eq:vertical-sections}
 \Omega_{n,k}(m):=\{M:(m,M)\in\Omega_{n,k}\}
 =\begin{cases}
 [\phi_k,\pi),&0<m\leq\gamma_k,\\
 [\phi_k,\beta_k(m)],&\gamma_k<m\leq\theta_k,\\
 \varnothing,&m>\theta_k.
 \end{cases}
\end{equation}
These sections lie in $D$, since $m\leq\theta_k<\phi_k$. If $m\leq\gamma_k$, then $U_k(m)\geq\pi$, so the unbounded-supremum alternative in \eqref{eq:branchwise} always applies. If $\gamma_k<m\leq\theta_k$, that alternative covers $M\leq U_k(m)$, and strict decrease of $g_k$ covers exactly $U_k(m)<M\leq\beta_k(m)$. Hence \cref{cor:branchwise} gives
\begin{equation}\label{eq:sector-union}
 \Sigma_n\cap\Cp=\Lambda(\Omega_n),
 \qquad \Omega_n=\bigcup_{k=1}^K\Omega_{n,k}.
\end{equation}
This proves the complete branch geometry before the polynomial equation for the arcs is introduced.

\subsection{Why the lower edge is a uniform chord}\label{sec:uniform}

On $M=\phi_k$, the minimum in \eqref{eq:minimum} is zero. Strict convexity forces every feasible zero on that argument branch to have $u_1=\cdots=u_n=\theta_k$. The inverse map \eqref{eq:inverse-parameter} then makes all loop parameters equal.

Let $P_n$ be the cyclic permutation matrix and define
\begin{equation}\label{eq:uniform-family}
 \Un(q)=(1-q)I+qP_n=A_n(1-q,\ldots,1-q),\qquad 0<q\leq1.
\end{equation}
Its eigenvalues are $1-q+q\ee^{2\pi\ii j/n}$, $0\leq j<n$. In particular, the $k$th upper branch traces the chord from $1$ to $\omega_k$. The endpoint $1$ is its limit as $q\downarrow0$, but is also the Perron eigenvalue of every admissible matrix.

The chord parameter can be read directly from the angular coordinates. Since
\[
 \omega_k-1=2\sin(\theta_k/2)\ee^{\ii\phi_k},
\]
identity~\eqref{eq:ray-identity} gives, for $0<m\leq\theta_k$,
\begin{equation}\label{eq:chord-parameter}
 \begin{gathered}
 \Lambda(m,\phi_k)=1-q_k(m)+q_k(m)\omega_k,\\
 q_k(m)=\frac{\sin m}{2\sin(\theta_k/2)\sin(\phi_k-m)}\in(0,1].
 \end{gathered}
\end{equation}
The range follows also from the fact that $\sin m/\sin(\phi_k-m)$ increases from $0$ to $|\omega_k-1|$. Equivalently, the parameter is $q=(\lambda-1)/(\omega_k-1)$. Thus the entire lower edge is realized, regardless of which part remains visible after taking the union of sectors.

\subsection{Why the upper edge has only one loop}

On $M=\beta_k(m)$, the finite maximum is zero. Its vertex vector is
\[
 (U_k(m),m,\ldots,m).
\]
Under \eqref{eq:inverse-parameter}, every coordinate equal to $m$ becomes a zero loop parameter. This identifies the second extremal family without having to guess a boundary polynomial.

\begin{proposition}[One-loop realization and polynomial parametrization]\label{prop:one-loop}
For $\gamma_k<m\leq\theta_k$, put
\begin{equation}\label{eq:zeta-k}
 \zeta_k(m)=\Lambda(m,\beta_k(m)).
\end{equation}
Then $\zeta_k(m)$ is an eigenvalue of $\AL(\alpha)=A_n(\alpha,0,\ldots,0)$ for a parameter $\alpha\in[0,1)$, namely
\begin{equation}\label{eq:alpha-k}
 \alpha_k(m)=\frac{\zeta_k(m)^n-1}{\zeta_k(m)^{n-1}-1}.
\end{equation}
Moreover, $\rho_k(m)=|\zeta_k(m)|$ satisfies \eqref{eq:radial-polynomial}. It extends continuously with $\rho_k(\gamma_k)=0$ and $\rho_k(\theta_k)=1$, and is the unique solution of that equation in $[0,1]$.
\end{proposition}

\begin{proof}
Write $\lambda=a+\ii b=\zeta_k(m)$, $M=\beta_k(m)$, and $U=U_k(m)$. The vector $(U,m,\ldots,m)$ is feasible and its logarithmic sum is $g_k(m,M)=0$. By \cref{prop:feasibility}, it realizes $\lambda$ with
\[
 \alpha=a-b\cot U\in[0,1),\qquad \alpha_2=\cdots=\alpha_n=0.
\]
The parameter bound follows from $m\leq U<M$ and the bijection in \cref{lem:angles}; it does not follow merely from the realness of a rational expression.

The product equation for this matrix is
\begin{equation}\label{eq:one-loop-polynomial}
 \lambda^{n-1}(\lambda-\alpha)=1-\alpha,
 \qquad\text{or}\qquad
 p_\alpha(\lambda)=\lambda^n-\alpha\lambda^{n-1}+\alpha-1=0.
\end{equation}
If $\lambda^{n-1}=1$ in this equation, then $\lambda=1$, which is impossible in $\Cp$. Solving for $\alpha$ therefore gives \eqref{eq:alpha-k}.

To obtain the radial equation, write $\rho=|\lambda|=\sin M/\sin(M-m)$. The trigonometric identity
\begin{equation}\label{eq:positive-denominator}
 \sin U-\rho\sin(U-m)
 =\frac{\sin m\sin(M-U)}{\sin(M-m)}>0
\end{equation}
gives
\begin{equation}\label{eq:exp-g}
 \ee^{g_k(m,M)}=\frac{\rho^n\sin m}{\sin U-\rho\sin(U-m)}.
\end{equation}
At $g_k=0$ this becomes
\begin{equation}\label{eq:radial-U}
 \rho^n\sin m+\rho\sin(U-m)=\sin U,
\end{equation}
which is \eqref{eq:radial-polynomial}. For $\gamma_k<m<\theta_k$, both $\sin m$ and $\sin(U-m)$ are positive, so the left-hand side is strictly increasing for $\rho\geq0$. Our realizing matrix already gives a solution with $0<\rho\leq1$, and hence the solution in $[0,1]$ is unique.

At $m=\theta_k$, one has $U=m$ and \eqref{eq:radial-U} forces $\rho=1$. At $m=\gamma_k$, it has $U=\pi$ and its only nonnegative solution is $\rho=0$. Continuity at the first endpoint follows from \cref{prop:graph}. At the second, $M\to\pi$ and $m\to\gamma_k\in(0,\pi)$ in \eqref{eq:Lambda}, so $\rho\to0$. Thus $\zeta_k$ extends continuously by $\zeta_k(\gamma_k)=0$.
\end{proof}

This also proves that the graph constructed from $g_k=0$ is exactly the graph defined in \eqref{eq:beta-definition}, so the notation in \cref{thm:region} is now justified. Its one-loop parameter satisfies
\begin{equation}\label{eq:alpha-endpoints}
 \alpha_k(\theta_k)=0,\qquad
 \lim_{m\downarrow\gamma_k}\alpha_k(m)=1.
\end{equation}
The first identity follows from $U=m$; the second follows from \eqref{eq:alpha-k} and $\zeta_k(m)\to0$. Thus $0$ is not attained by the one-loop family itself.

For completeness, the derivative in \eqref{eq:implicit-beta-derivative} has a radial form. With $\rho=\rho_k(m)$, $M=\beta_k(m)$, $U=U_k(m)$, and $V=U-m$, cotangent differences in \eqref{eq:g-M}--\eqref{eq:g-m} give
\[
 \beta'_k(m)=-\frac{(n-1)\sin V\sin M}
 {(n-1)\sin m\sin(M-U)+\sin U\sin(M-m)}.
\]
At $g_k=0$, $\sin(M-U)=\rho^{n-1}\sin M$ and $\sin(M-m)=\sin M/\rho$. Substituting these identities and then \eqref{eq:radial-U} yields, for $\gamma_k<m<\theta_k$,
\begin{equation}\label{eq:radial-beta-derivative}
 \beta'_k(m)=-\frac{(n-1)\sin(2\pi k-nm)}
 {n\rho_k(m)^{n-1}\sin m+\sin(2\pi k-nm)}<0.
\end{equation}
The sign argument itself did not require this simplification.

At the endpoint $(\theta_k,\phi_k)$, the graph $M=\beta_k(m)$ is tangent to the horizontal line $M=\phi_k$. More precisely, the one-sided endpoint derivatives are
\begin{equation}\label{eq:activation-derivatives}
 \beta_k'(\theta_k)=0,\qquad
 \beta_k''(\theta_k)=\frac{n-1}{\sin\theta_k}.
\end{equation}
To see this without differentiating an inverse function at its endpoint, put $\theta=\theta_k$ and $\phi=\phi_k$. The logarithmic function $g_k$ is real analytic near $(\theta,\phi)$, since all its sine arguments are positive there. At that point $U_k(\theta)=\theta$, and \eqref{eq:g-M}--\eqref{eq:g-m} give
\[
 \partial_Mg_k=-2n\tan(\theta/2)\ne0,
 \qquad \partial_mg_k=0,
 \qquad \partial_{mm}g_k=n(n-1)\sec^2(\theta/2).
\]
The implicit function theorem gives a local analytic extension of $\beta_k$. Differentiating $g_k(m,\beta_k(m))=0$ once and twice, and using $\beta_k'(\theta)=0$ in the second identity, yields \eqref{eq:activation-derivatives}. Hence, for $\delta\downarrow0$,
\[
 \beta_k(\theta_k-\delta)-\phi_k
 =\frac{n-1}{2\sin\theta_k}\delta^2+O(\delta^3).
\]
Inverting this expansion on the decreasing branch gives
\begin{equation}\label{eq:inverse-activation-expansion}
 h_k(M)=\theta_k-
 \sqrt{\frac{2\sin\theta_k}{n-1}}\sqrt{M-\phi_k}
 +O(M-\phi_k),\qquad M\downarrow\phi_k.
\end{equation}
Here $M$ approaches from above, and the constants may depend on $n$ and $k$. The inverse branch has a square-root change as $M$ decreases to $\phi_k$. The global function $\mu_n$ also changes branch at the levels $M=\phi_j$ for $2\leq j\leq K$, so \eqref{eq:inverse-activation-expansion} describes its behavior immediately above such a level, not across the jump from below.

\section{Horizontal sections determine the boundary}\label{sec:boundary}

The sectors in \eqref{eq:vertical-sections} overlap. Instead of testing each of their edges against every other sector, we determine the right endpoint at each fixed height $M$. \Cref{fig:horizontal} depicts this change of viewpoint.

\begin{proposition}[Ordered horizontal sections]\label{prop:horizontal}
For $\phi_k\leq M<\pi$, the $k$th sector has horizontal section
\begin{equation}\label{eq:branch-section}
 \{m:(m,M)\in\Omega_{n,k}\}=(0,h_k(M)],\qquad h_k=\beta_k^{-1}.
\end{equation}
If $k<\ell$ and both sectors are active at height $M$, then $h_k(M)<h_\ell(M)$. Consequently,
\begin{equation}\label{eq:union-section}
 I(M):=\{m:(m,M)\in\Omega_n\}
 =\begin{cases}
 \varnothing,&0<M<\phi_1,\\
 (0,\mu_n(M)],&\phi_1\leq M<\pi.
 \end{cases}
\end{equation}
\end{proposition}

\begin{proof}
Strict decrease of $\beta_k$ changes the inequality $M\leq\beta_k(m)$ into $m\leq h_k(M)$ when $m>\gamma_k$. Since $h_k(M)>\gamma_k$, this joins the rectangular part $0<m\leq\gamma_k$ of the sector to give \eqref{eq:branch-section}.

The branch intervals are separated:
\begin{equation}\label{eq:separation}
 \gamma_k<\theta_k<\gamma_{k+1}\qquad(k<K).
\end{equation}
Indeed,
\[
 \theta_k-\gamma_k=\frac{(n-2k)\pi}{n(n-1)}>0,
 \qquad \gamma_{k+1}-\theta_k=\frac{(n+2k)\pi}{n(n-1)}>0.
\]
Therefore, whenever adjacent sectors are active,
\[
 h_k(M)\leq\theta_k<\gamma_{k+1}<h_{k+1}(M).
\]
Iterating proves the ordering for all active indices. Taking their union gives \eqref{eq:union-section}: the largest active index always has the largest section.
\end{proof}

Equations~\eqref{eq:sector-union} and \eqref{eq:union-section} prove the nonreal membership assertion in \cref{thm:region}. Set $\phi_{K+1}=\pi$. Then
\begin{equation}\label{eq:mu-bands}
 \mu_n(M)=h_k(M)\qquad(\phi_k\leq M<\phi_{k+1}).
\end{equation}
The endpoint function is strictly decreasing on each such band and has upward jumps at $M=\phi_k$ for $2\leq k\leq K$. The boundary consists of its graph pieces and the horizontal segments supplied by those jumps.

For $1\leq k<K$, define
\begin{equation}\label{eq:junctions}
 \eta_k=h_k(\phi_{k+1}),\qquad
 p_k=\Lambda(\eta_k,\phi_{k+1})=\zeta_k(\eta_k).
\end{equation}
Because $\phi_k<\phi_{k+1}<\pi$, one has $\gamma_k<\eta_k<\theta_k$. Equivalently, $\eta_k$ is the unique solution of $\beta_k(\eta_k)=\phi_{k+1}$. In terms of the endpoint function,
\begin{equation}\label{eq:mu-jump}
 \lim_{M\uparrow\phi_{k+1}}\mu_n(M)=\eta_k,
 \qquad \mu_n(\phi_{k+1})=\theta_{k+1}.
\end{equation}

\begin{figure}[tbp]
 \centering
 \includegraphics[width=0.91\textwidth]{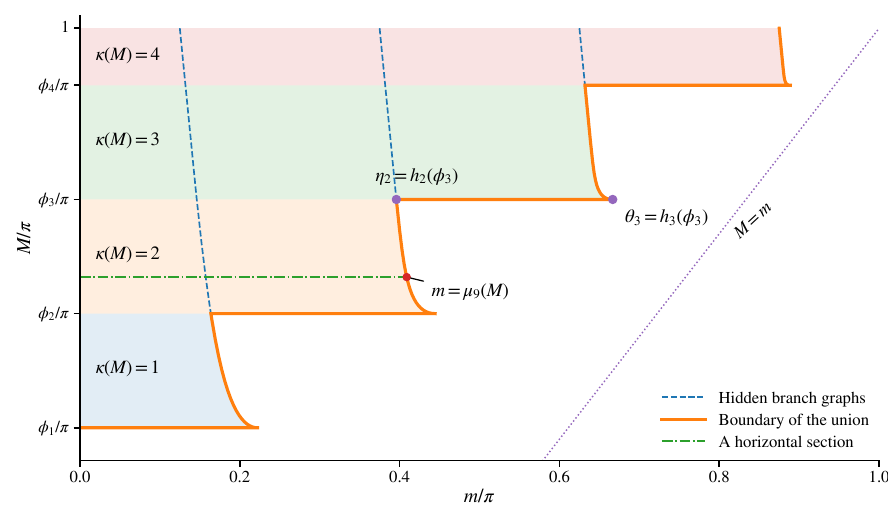}
 \caption{Horizontal sections of $\Omega_9$. In each band $\phi_k\leq M<\phi_{k+1}$, the section is $(0,\mu_9(M)]$, with $\mu_9(M)=h_k(M)$. Dashed curves show portions of full branch graphs hidden inside later sectors. The solid boundary follows one graph at a time and has a horizontal segment at each level $M=\phi_k$. At $M=\phi_3$, the boundary segment runs from $\eta_2=h_2(\phi_3)$ to $\theta_3=h_3(\phi_3)$. The sample horizontal section extends from the excluded axis $m=0$ to the endpoint $m=\mu_9(M)$. Neither $M=\pi$ nor $m=0$ belongs to $D$.}
 \label{fig:horizontal}
\end{figure}

For complex numbers $z,w$, let $[z,w]$ denote their closed straight segment. Set
\begin{align}
 C_1&=[1,\omega_1],&
 C_k&=[p_{k-1},\omega_k]&& (2\leq k\leq K),\label{eq:visible-chords}\\
 L_k&=\{\zeta_k(m):\eta_k\leq m\leq\theta_k\}&&& (k<K),\notag\\[-2pt]
 L_K&=\{\zeta_K(m):\gamma_K\leq m\leq\theta_K\}.&&&&\label{eq:visible-arcs}
\end{align}
The definition of $L_K$ uses the continuous extension $\zeta_K(\gamma_K)=0$. Put
\begin{equation}\label{eq:boundary-chains}
 B_n^+=\bigcup_{k=1}^K(C_k\cup L_k),
 \qquad B_n^-=\{\overline z:z\in B_n^+\}.
\end{equation}
Conjugation is thus kept distinct from the closure notation $\cl\Sigma_n$ used below.

\begin{theorem}[Boundary chain]\label{thm:boundary}
The nonreal upper boundary is $B_n^+\cap\Cp$. The set $B_n^+$ is a simple chain from $1$ to $0$, alternating in the order
\[
 C_1,L_1,C_2,L_2,\ldots,C_K,L_K.
\]
The initial chord runs from $1$ to $\omega_1$, and the final arc from $\omega_K$ to $0$. Each intervening pair has the orientation
\begin{equation}\label{eq:oriented-chain}
 \omega_k\xrightarrow{L_k}p_k\xrightarrow{C_{k+1}}\omega_{k+1}
 \qquad(1\leq k<K).
\end{equation}
When $K=1$, there are no intervening pairs and the chain is $1\xrightarrow{C_1}\omega_1\xrightarrow{L_1}0$.
The full topological boundary is
\begin{equation}\label{eq:full-boundary}
 \partial\Sigma_n=B_n^+\cup B_n^-\cup
 \begin{cases}
 [-1,0],&n\text{ even},\\
 \varnothing,&n\text{ odd}.
 \end{cases}
\end{equation}
Every nonreal boundary point is attained by an admissible matrix. The chords are realized by \eqref{eq:uniform-family}, and the arcs away from $0$ by the one-loop family in \cref{prop:one-loop}.
\end{theorem}

\begin{proof}[Proof of the nonreal boundary assertion]
Take the boundary relative to the open angular domain $D$. Each $\Omega_{n,k}$ is relatively closed: in \eqref{eq:branch-section}, $M\geq\phi_k$ and $m\leq h_k(M)$ are closed conditions, with $h_k$ continuous for $\phi_k\leq M<\pi$. Thus $\Omega_n$, a finite union, is also relatively closed.

Below $M=\phi_1$, all points of $D$ are exterior. On an open band $\phi_k<M<\phi_{k+1}$, \eqref{eq:union-section} is $0<m\leq\mu_n(M)=h_k(M)$. By continuity, points with $m<h_k(M)$ are interior, points with $m>h_k(M)$ are exterior, and precisely the graph $m=h_k(M)$ is boundary.

At the first level $M=\phi_1$, there are no points of $\Omega_n$ below the level and the section on it is $(0,\theta_1]$. This gives the initial horizontal boundary. At a later level $M=\phi_{k+1}$, \eqref{eq:mu-jump} gives the limiting endpoints $\eta_k$ from below and $\theta_{k+1}$ on the level and from above. Points with $m<\eta_k$ are interior, since both neighboring sections extend strictly beyond them. Points with $m>\theta_{k+1}$ are exterior. For $\eta_k<m<\theta_{k+1}$, points just below the level are outside and points just above it are inside. Hence the open interval between these endpoints is boundary. Its endpoints are boundary as well, because the relative boundary is closed. These band and level cases exhaust $D$; there are no other boundary pieces.

In the original graph parameter, the visible part of the $k$th upper edge is therefore $\eta_k\leq m\leq\theta_k$ for $k<K$, and the entire edge $\gamma_K<m\leq\theta_K$ for $k=K$. The horizontal pieces map to \eqref{eq:visible-chords}, by \eqref{eq:chord-parameter}. The graph pieces map to \eqref{eq:visible-arcs}, by \eqref{eq:zeta-k}. Their missing angular endpoints have the plane limits
\[
 \lim_{m\downarrow0}\Lambda(m,\phi_1)=1,
 \qquad \lim_{m\downarrow\gamma_K}\zeta_K(m)=0.
\]
The homeomorphism $\Lambda$ preserves relative boundaries. It also shows that the pieces form a simple chain: $M$ increases along successive arcs, and each constant-$M$ piece is a single segment. The only real endpoints of this chain are $1$ and $0$. All its nonreal points lie in $\Lambda(\Omega_n)$ and hence are attained. The assertions on realizing families follow from \cref{prop:one-loop} and \eqref{eq:chord-parameter}. The real-axis part of \eqref{eq:full-boundary} is proved in \cref{sec:real-boundary}.
\end{proof}

At a fixed height, all active sections start at $m=0$ and their right endpoints are strictly ordered. There is only one outer endpoint to examine; nonadjacent branches cannot contribute additional exposed pieces.

\section{The real axis, endpoint attainment, and radial realization}\label{sec:real}

\subsection{Real eigenvalues are attained explicitly}

\begin{proposition}\label{prop:real}
Equation~\eqref{eq:real-set} holds, including $(0,1)\subset\interior\Sigma_n$.
\end{proposition}

\begin{proof}
Every real eigenvalue belongs to $[-1,1]$. The product equation requires
\begin{equation}\label{eq:real-product}
 \prod_{j=1}^n(r-\alpha_j)=\prod_{j=1}^n(1-\alpha_j)>0.
\end{equation}
When $n$ is odd and $r<0$, the left side is negative; for $r=0$ it is nonpositive. Thus the odd-dimensional real set is contained in $(0,1]$.

To attain a prescribed $r\in(0,1)$ in any dimension $n\geq3$, put
\begin{equation}\label{eq:positive-real-realization}
 c=r^{(n-2)/2},\qquad a=\frac{1+rc}{1+c},
 \qquad A=A_n(a,a,0,\ldots,0).
\end{equation}
Then $r<a<1$ and
\[
 a-r=\frac{1-r}{1+c},\qquad 1-a=\frac{c(1-r)}{1+c}.
\]
Consequently,
\[
 (r-a)^2r^{n-2}=(1-a)^2,
\]
which is exactly \eqref{eq:real-product} for this matrix. The value $1$ is always attained.

When $n$ is even, $-1$ is an eigenvalue of $P_n$. The corresponding branch of $\Un(q)$ is $1-2q$, which runs through $[-1,1)$ as $q$ runs through $(0,1]$. Together with the Perron eigenvalue, this gives all of $[-1,1]$; a prescribed $r<1$ is obtained with $q=(1-r)/2$.

It remains to prove interiority, not just real-axis membership. Fix $r\in(0,1)$. For upper-half-plane points approaching $r$, one has $m\to0$ and $M\to\pi$. A sufficiently small neighborhood of $r$ therefore has
\[
 0<m<\gamma_1,\qquad \phi_1<M<\pi
\]
throughout its upper half. These points belong to the rectangular part of $\Omega_{n,1}$ in \eqref{eq:vertical-sections}. The lower half belongs by conjugation, and the real diameter belongs by \eqref{eq:positive-real-realization}. Hence a whole open disk about $r$ lies in $\Sigma_n$.
\end{proof}

\subsection{Completing the topological boundary}\label{sec:real-boundary}

For every upper-half-plane spectral point, \eqref{eq:vertical-sections} gives
\begin{equation}\label{eq:argument-bound}
 \Arg\lambda\leq\theta_K<\pi.
\end{equation}
Near a fixed negative real number, upper-half-plane arguments instead tend to $\pi$. Thus a sufficiently small neighborhood of that real number contains no nonreal spectral points. In even dimension, the attained interval $[-1,0)$ is therefore boundary, not interior. In odd dimension, negative real numbers do not even belong to $\cl\Sigma_n$.

The final arc approaches $0$, so $0\in\cl\Sigma_n$. Points with arguments strictly between $\theta_K$ and $\pi$ approach $0$ from outside $\Sigma_n$, by \eqref{eq:argument-bound}; hence $0\in\partial\Sigma_n$. \Cref{prop:real} shows that it is attained exactly when $n$ is even. In that case $\Un(1/2)$ realizes it; it is never attained by an admissible one-loop matrix because $p_\alpha(0)=\alpha-1\ne0$.

The point $1$ is attained and is boundary because $|\lambda|\leq1$. The interval $(0,1)$ is interior by \cref{prop:real}, and no real point outside $[-1,1]$ is a limit point. These observations, together with the nonreal boundary calculation, prove \eqref{eq:full-boundary} and complete \cref{thm:region,thm:boundary}. They also give the precise closure statement
\begin{equation}\label{eq:closure}
 \cl\Sigma_n=\Sigma_n\cup\{0\}.
\end{equation}
Indeed, the nonreal region is relatively closed by the proof of \cref{thm:boundary}, and all possible real limit points have just been classified. Thus the spectral set is closed in even dimension; in odd dimension, $0$ is its only missing limit point.

\begin{corollary}[Allowing deleted cycle edges]\label{cor:closed-domain}
For $n\geq3$, let
\[
 \widetilde\Sigma_n=\bigcup_{\alpha\in[0,1]^n}\sigma(A_n(\alpha)).
\]
Then $\widetilde\Sigma_n=\Sigma_n\cup\{0\}=\cl\Sigma_n$.
\end{corollary}

\begin{proof}
The two-permutation expansion proving \eqref{eq:determinant} remains valid for parameters in $[0,1]^n$. If some $\alpha_j=1$, the cycle product vanishes and
\[
 \det(\lambda I-A_n(\alpha))=\prod_{j=1}^n(\lambda-\alpha_j).
\]
Every eigenvalue of such a matrix is therefore real and belongs to $[0,1]$. Conversely, $0$ is attained with parameters $(1,0,\ldots,0)$, while the matrices with all parameters strictly below $1$ already give $\Sigma_n$. Since $(0,1]\subset\Sigma_n$ by \cref{prop:real}, these observations show that the extended union is exactly $\Sigma_n\cup\{0\}$. Equation~\eqref{eq:closure} identifies this set with $\cl\Sigma_n$. The normalized equation \eqref{eq:normalized-product} continues to be used only when every cycle edge is positive, as its division is unavailable at a deleted edge.
\end{proof}

\subsection{An explicit realizing matrix for every nonreal point}

The section formula does more than locate the boundary. It chooses a boundary point beyond any given nonreal spectral point on the same ray from $1$. This makes \eqref{eq:affine-realization} constructive.

\begin{corollary}[Radial reconstruction]\label{cor:reconstruction}
Let $\lambda=\Lambda(m,M)\in\Sigma_n\cap\Cp$, put $\mu=\mu_n(M)$, and define
\begin{equation}\label{eq:radial-reconstruction}
 \zeta=\Lambda(\mu,M),\qquad
 \alpha_* =\frac{\zeta^n-1}{\zeta^{n-1}-1},\qquad
 s=\frac{\sin m\sin(M-\mu)}{\sin\mu\sin(M-m)}.
\end{equation}
Then $\alpha_*\in[0,1)$, $s\in(0,1]$, and
\begin{equation}\label{eq:reconstructed-matrix}
 \lambda\in\sigma\bigl(A_n(1-s+s\alpha_*,1-s,\ldots,1-s)\bigr).
\end{equation}
The same real matrix realizes $\overline\lambda$. At $M=\phi_k$, the construction has $\zeta=\omega_k$ and $\alpha_*=0$, so it reduces to the uniform family.
\end{corollary}

\begin{proof}
For $k=\kappa(M)$ one has $\mu=h_k(M)$, so $\zeta$ lies on the graph $M=\beta_k(\mu)$. Thus \cref{prop:one-loop} gives its admissible one-loop parameter $\alpha_*$. By \eqref{eq:ray-identity}, distance along a fixed ray is
\[
 R(m,M)=\frac{\sin m}{\sin(M-m)},\qquad
 \partial_mR(m,M)=\frac{\sin M}{\sin^2(M-m)}>0.
\]
Since membership means $0<m\leq\mu$, the ratio $s=R(m,M)/R(\mu,M)$ belongs to $(0,1]$, and $\lambda=1+s(\zeta-1)$. Applying \eqref{eq:affine-realization} to $\AL(\alpha_*)$ yields \eqref{eq:reconstructed-matrix}.
\end{proof}

Writing $b=1-s$ and $a=1-s+s\alpha_*$ in \eqref{eq:reconstructed-matrix} makes the two-parameter form explicit. Indeed,
\[
 0\leq b\leq a,\qquad
 1-a=s(1-\alpha_*)>0.
\]
Thus $0\leq b\leq a<1$, including when $s=1$ or $\alpha_*=0$. The interpolation assigns the common holding probability $b$ to every stage and adds $s\alpha_*$ at one distinguished stage. No cycle edge is lost in this construction.

Every upper-half-plane spectral point is therefore obtained from a one-loop boundary point by an admissible affine interpolation with $I$. In set form,
\begin{equation}\label{eq:radial-set}
 \begin{gathered}
 \Sigma_n\cap\Cp
 =\{1+s(\zeta(M)-1):\phi_1\leq M<\pi,\ 0<s\leq1\},\\
 \zeta(M)=\Lambda(\mu_n(M),M).
 \end{gathered}
\end{equation}
The positive real construction \eqref{eq:positive-real-realization}, the even-dimensional branch $1-2q$, and the always-present eigenvalue $1$ cover the remaining attained points.

\subsection{Realization summary}\label{sec:dictionary}

\Cref{tab:realization} collects the formulas needed to realize a prescribed spectral point. For a lower-half-plane point, apply the corresponding upper-half-plane row to its conjugate; the resulting real matrix realizes both. At a switching point $p_k$, both $\AL(\alpha_k(\eta_k))$ and the uniform family on $C_{k+1}$ realize the same eigenvalue. At $\omega_k$, both families give $P_n$. A realizing matrix therefore need not be unique at a junction.

\begin{table}[!htbp]
\centering
\caption{Realization formulas for $n\geq3$. The symbols $\Un$ and $\AL$ denote the uniform and one-loop families. In the radial row, $(m,M)$ are the coordinates of $\lambda$ and $\mu=\mu_n(M)$. Every displayed parameter is in its admissible range.}
\label{tab:realization}
\small
\renewcommand{\arraystretch}{1.35}
\begin{tabularx}{\textwidth}{@{}p{0.20\textwidth}p{0.28\textwidth}X@{}}
\toprule
Spectral point & Realizing matrix & Parameter formula \\
\midrule
$\lambda=1$ & Any $A_n(\alpha)$ & In particular, $P_n=A_n(0,\ldots,0)$. \\
\addlinespace[2pt]
$\lambda\in C_k\setminus\{1\}$ & $\Un(q)$ & $q=\dfrac{\lambda-1}{\omega_k-1}\in(0,1]$. \\
\addlinespace[2pt]
$\lambda\in L_k\setminus\{0\}$ & $\AL(\alpha)$ & $\alpha=\dfrac{\lambda^n-1}{\lambda^{n-1}-1}\in[0,1)$. \\
\addlinespace[2pt]
Any $\lambda\in\Sigma_n\cap\Cp$ & $(1-s)I+s\AL(\alpha_*)$ & $\zeta=\Lambda(\mu,M)$,\newline $\alpha_*=(\zeta^n-1)/(\zeta^{n-1}-1)$,\newline $s=|\lambda-1|/|\zeta-1|\in(0,1]$. \\
\addlinespace[2pt]
$r\in(0,1)$ & $A_n(a,a,0,\ldots,0)$ & $c=r^{(n-2)/2}$,\quad $a=\dfrac{1+rc}{1+c}$. \\
\addlinespace[2pt]
$r\in[-1,0]$, $n$ even & $\Un(q)$ & $q=(1-r)/2\in[1/2,1]$. \\
\bottomrule
\end{tabularx}
\par\smallskip
\begin{minipage}{\textwidth}\small
For even $n$, the last row includes $0$, realized by $\Un(1/2)$. For odd $n$, $0$ is a boundary point but is not an eigenvalue of any admissible matrix.
\end{minipage}
\end{table}

\section{Algebraic criteria and one-loop continuation}\label{sec:algebra}

This section relates the geometric proof to the real algebraic equation and identifies the full one-loop eigenvalue branches. These statements supplement, rather than enter into, the proof of the section formula.

\subsection{A polynomial equation for the arcs}

Define
\begin{equation}\label{eq:H}
 H_n(\lambda)=\Imag\bigl((\lambda^n-1)\overline{(\lambda^{n-1}-1)}\bigr).
\end{equation}
For $\lambda=a+\ii b$, introduce the reduced polynomial
\begin{equation}\label{eq:G}
 G_n(a,b)=\frac{H_n(a+\ii b)}{b\bigl((a-1)^2+b^2\bigr)},
\end{equation}
where the quotient is extended polynomially at its apparent singularities.

\begin{proposition}[Algebraic equation and membership criterion]\label{prop:carrier}
The extension in \eqref{eq:G} is a real polynomial of degree $2n-4$. Every full one-loop arc lies on $H_n=0$, equivalently on $G_n=0$ away from the real axis. In the finite-maximum regime of branch $k$,
\begin{equation}\label{eq:sign-equivalence}
 \begin{gathered}
 m\leq\theta_k<M,\quad U_k(m)<M\quad\Longrightarrow\\
 \bigl(g_k(m,M)\geq0\ \Longleftrightarrow\ H_n(\Lambda(m,M))\geq0\bigr).
 \end{gathered}
\end{equation}
Consequently, for an upper-half-plane point with $M\geq\phi_1$, let $k=\kappa(M)$ be the largest active branch. The point is spectral exactly when this $k$ satisfies
\begin{equation}\label{eq:algebraic-membership}
 m\leq\theta_k,\qquad M\geq\phi_k,\qquad
 \bigl(M\leq U_k(m)\ \text{or}\ H_n(\lambda)\geq0\bigr).
\end{equation}
\end{proposition}

\begin{proof}
Factoring $\lambda^j-1=(\lambda-1)(1+\cdots+\lambda^{j-1})$ in \eqref{eq:H} gives
\[
 H_n(\lambda)=|\lambda-1|^2\Imag\left(
 (1+\cdots+\lambda^{n-1})\overline{(1+\cdots+\lambda^{n-2})}
 \right).
\]
The imaginary part is a real polynomial in $a,b$ that changes sign under $b\mapsto-b$, so it is divisible by $b$. Its highest-degree term is $b(a^2+b^2)^{n-2}$. This proves polynomial extension and degree $2n-4$ for $G_n$.

Whenever $\lambda^{n-1}\ne1$, the equation $H_n(\lambda)=0$ is equivalent to the realness of $(\lambda^n-1)/(\lambda^{n-1}-1)$. \Cref{prop:one-loop} therefore places each one-loop arc in the zero set $H_n=0$. Its endpoint $0$ lies on $H_n=0$ by continuity.

For the sign assertion, put $\lambda=\rho\ee^{\ii m}$ and $U=U_k(m)$. In the stated regime the denominator in \eqref{eq:exp-g} is strictly positive by \eqref{eq:positive-denominator}. Hence
\[
 g_k(m,M)\geq0
 \quad\Longleftrightarrow\quad
 \rho^n\sin m+\rho\sin(U-m)-\sin U\geq0.
\]
Expanding \eqref{eq:H}, and using $U=2\pi k-(n-1)m$, gives
\begin{align*}
 H_n(\lambda)
 &=\rho^{2n-1}\sin m-\rho^n\sin(nm)+\rho^{n-1}\sin((n-1)m)\\
 &=\rho^{n-1}\bigl[\rho^n\sin m+\rho\sin(U-m)-\sin U\bigr].
\end{align*}
Since $\rho>0$, the signs agree. Substitution into \cref{cor:branchwise} gives the criterion on each argument branch. By \cref{prop:horizontal}, the section of the largest active branch contains the sections of all earlier active branches. It therefore suffices to apply the criterion to $k=\kappa(M)$, which proves \eqref{eq:algebraic-membership}.
\end{proof}

\begin{proof}[Proof of \cref{cor:direct-membership}]
The exclusion below $\phi_1$ is already part of \cref{thm:region}. For $\phi_1\leq M<\pi$, the identity $\phi_j=\pi/2+\pi j/n$ gives
\[
 \phi_j\leq M
 \quad\Longleftrightarrow\quad
 j\leq n\left(\frac{M}{\pi}-\frac12\right).
\]
The expression on the right lies in $[1,n/2)$, so its floor belongs to $\{1,\ldots,K\}$. This proves \eqref{eq:active-floor}, including the convention that branch $j$ is active at $M=\phi_j$. In \eqref{eq:algebraic-membership}, the inequality $M\geq\phi_k$ is automatic for this index. Removing that redundant condition gives \eqref{eq:direct-membership}. The sign comparison in \eqref{eq:sign-equivalence} is invoked only when $M>U_k(m)$; at $M=U_k(m)$ the unbounded-supremum alternative of \cref{prop:range} applies. Conjugation handles the lower half-plane, and \cref{prop:real} handles the real axis.
\end{proof}

In dimension four, the reduced polynomial can be written explicitly:
\begin{equation}\label{eq:four-cycle-specialization}
 \begin{split}
 G_4(a,b)
 &=(a^2+b^2)^2+2a(a^2+b^2)+3a^2-b^2\\
 &=(a^2+b^2+a)^2+2a^2-b^2.
 \end{split}
\end{equation}
This is the defining polynomial in the four-cycle description of Vagenende, Verbeken, Algaba, and Guerry~\cite{VagenendeEtAl}. The general construction retains that polynomial in dimension four and adds the branch selection needed in arbitrary dimension.

The polynomial equation alone does not describe the boundary. The realness condition must be combined with the admissibility of the parameter and the requirement that the point lie on the boundary of the union. The full one-loop arcs require $\gamma_k<m\leq\theta_k$ and $M=\beta_k(m)$; the actual boundary additionally requires the truncations in \eqref{eq:visible-arcs}. Likewise, \eqref{eq:sign-equivalence} is used only where its denominator has been proved positive. These restrictions prevent extra algebraic components from entering the spectral description.

\subsection{The full arcs are the actual one-loop branches}

\begin{proposition}\label{prop:continuation}
For $0\leq\alpha<1$, all roots of $p_\alpha$ in \eqref{eq:one-loop-polynomial} are simple. Every upper-half-plane one-loop eigenvalue lies on exactly one of the full arcs
\[
 L_k^{\mathrm{full}}=\{\zeta_k(m):\gamma_k<m\leq\theta_k\},
 \qquad 1\leq k\leq K.
\]
The branch starting at $\omega_k$ when $\alpha=0$ remains in $\Cp$, traces the full arc $L_k^{\mathrm{full}}$, and tends to $0$ as $\alpha\uparrow1$.
\end{proposition}

\begin{proof}
First let $\lambda\in\Cp$ satisfy $p_\alpha(\lambda)=0$, and put $m=\Arg\lambda$, $M=\Arg(\lambda-1)$, and $u=\Arg(\lambda-\alpha)$. The normalized product equation gives
\[
 (n-1)m+u=2\pi k,\qquad (n-1)f_M(m)+f_M(u)=0
\]
for some $1\leq k\leq K$. Thus $u=U_k(m)$. Since $m\leq u<M<\pi$, one obtains $\gamma_k<m\leq\theta_k$, and the logarithmic equation is $g_k(m,M)=0$. Uniqueness in \cref{prop:graph} gives $M=\beta_k(m)$, so $\lambda\in L_k^{\mathrm{full}}$. The intervals $(\gamma_k,\theta_k]$ are disjoint by \eqref{eq:separation}, proving uniqueness of $k$. Conversely, every point of each full arc is realized by \cref{prop:one-loop}.

To control continuation, differentiate:
\[
 p'_\alpha(\lambda)=\lambda^{n-2}\bigl(n\lambda-(n-1)\alpha\bigr).
\]
A multiple root could only be $0$ or $x=(n-1)\alpha/n$. But
\[
 p_\alpha(0)=\alpha-1<0,\qquad
 p_\alpha(x)=-\frac{\alpha}{n}x^{n-1}+\alpha-1<0.
\]
Thus all roots are simple. They continue locally; boundedness by $1$ and simplicity allow continuation throughout every compact subinterval of $[0,1)$, and hence throughout $[0,1)$, without collisions. A nonreal branch cannot meet the real axis: at such a meeting it would coincide with its conjugate, contradicting simplicity. There are therefore exactly $K$ upper-half-plane branches, as there are at $\alpha=0$.

The integer $k$ in the argument equation is constant along each continued upper branch, because the individual arguments are continuous and their sum is an integer multiple of $2\pi$. At $\alpha=0$, that integer is $k$ for the root $\omega_k$. Hence there is exactly one upper root with each branch index at every $\alpha$, and the branch from $\omega_k$ stays on $L_k^{\mathrm{full}}$. Every point of that arc is realized, so it belongs to this unique continued branch. The branch therefore traces the whole arc.

Finally, factor off the Perron root:
\begin{equation}\label{eq:perron-factorization}
 p_\alpha(\lambda)=(\lambda-1)
 \left[\lambda^{n-1}+(1-\alpha)\sum_{j=0}^{n-2}\lambda^j\right].
\end{equation}
Every non-Perron root has $|\lambda|\leq1$ and consequently satisfies
\begin{equation}\label{eq:root-limit-bound}
 |\lambda|^{n-1}
 \leq(1-\alpha)\sum_{j=0}^{n-2}|\lambda|^j
 \leq(n-1)(1-\alpha).
\end{equation}
This proves convergence to $0$ and rules out convergence to the other root of the limiting polynomial.
\end{proof}

\section{Sequential-stage models and two extensions}\label{sec:models}

The cycle classification can be read in terms of the time required to pass through a sequence of stages. This interpretation also gives two exact ways to use the result beyond the original discrete-time cycle: an affine change of variables for continuous-time generators, and a quotient construction for a stochastic matrix with two parallel paths. We give the correspondences explicitly, including the restrictions under which they hold.

\subsection{Geometric holding times and the circuit equation}

Fix $A=A_n(\alpha)$ and write $q_j=1-\alpha_j\in(0,1]$. For its Markov chain $(X_t)_{t\geq0}$, a visit to stage $j$ ends at the first advance to $j+1$, with indices modulo $n$. The number $G_j$ of steps up to and including that advance has distribution
\[
 \mathbb P(G_j=t)=q_j\alpha_j^{t-1},\qquad t=1,2,\ldots.
\]
When $\alpha_j=0$, this means $G_j=1$ with probability one. The Markov property makes the successive holding times independent: once the chain enters the next stage, its subsequent transition choices have no dependence on the time already spent at earlier stages.

Starting at $X_0=1$, let $T$ be the time at which the chain first completes the circuit by traversing the edge $n\to1$. Then
\[
 T=G_1+\cdots+G_n.
\]
The same description applies between successive traversals of $n\to1$, so the circuit-completion epochs have independent, identically distributed increments. The loop parameters therefore specify a renewal model with a prescribed sequence of stages.

\begin{proposition}\label{prop:renewal}
Put
\[
 R_\alpha=
 \begin{cases}
  (\max_j\alpha_j)^{-1},&\max_j\alpha_j>0,\\
  +\infty,&\alpha_1=\cdots=\alpha_n=0.
 \end{cases}
\]
The probability generating function of $T$ is
\begin{equation}\label{eq:circuit-pgf}
 B_\alpha(z)=\mathbb E[z^T]
 =\prod_{j=1}^n\frac{q_jz}{1-\alpha_jz},\qquad |z|<R_\alpha.
\end{equation}
Its rational continuation satisfies
\begin{equation}\label{eq:renewal-determinant}
 \det(I-zA)
 =\prod_{j=1}^n(1-\alpha_jz)
  -z^n\prod_{j=1}^nq_j
 =\prod_{j=1}^n(1-\alpha_jz)\,[1-B_\alpha(z)],
\end{equation}
where the last expression is interpreted after clearing denominators at a pole. No eigenvalue of $A$ equals any $\alpha_j$. For $\lambda\ne0$ with $\lambda\ne\alpha_j$ for every $j$,
\begin{equation}\label{eq:spectral-renewal}
 \lambda\in\sigma(A)
 \quad\Longleftrightarrow\quad
 B_\alpha(\lambda^{-1})=1,
\end{equation}
using the rational continuation when $\lambda^{-1}$ lies outside the convergence disk in \eqref{eq:circuit-pgf}.
\end{proposition}

\begin{proof}
For prescribed holding times $t_1,\ldots,t_n$, the chain makes $t_j-1$ self-loops followed by one advance at stage $j$. The probability of that path is the product of the corresponding geometric probabilities, which proves the asserted independence. For $|\alpha_jz|<1$, summing the geometric series gives
\[
 \mathbb E[z^{G_j}]
 =\sum_{t=1}^\infty q_j\alpha_j^{t-1}z^t
 =\frac{q_jz}{1-\alpha_jz}.
\]
Independence gives \eqref{eq:circuit-pgf}. For $z\ne0$, multiplying the determinant identity in \cref{lem:product}, with $\lambda=z^{-1}$, by $z^n$ gives the first equality in \eqref{eq:renewal-determinant}. Both sides are polynomials, so the equality also holds at $z=0$. Factoring the product of denominators gives the second equality wherever those denominators are nonzero, and hence as a rational identity.

At $\lambda=\alpha_j$, the characteristic polynomial in \eqref{eq:determinant} has value $-\prod_iq_i<0$, so such a point cannot be an eigenvalue. For a nonzero $\lambda$ distinct from all $\alpha_j$, none of the factors $1-\alpha_j/\lambda$ vanishes. Substituting $z=\lambda^{-1}$ into \eqref{eq:renewal-determinant} and cancelling these factors proves \eqref{eq:spectral-renewal}.
\end{proof}

The distinction between a generating function and its rational continuation matters here. Formula~\eqref{eq:circuit-pgf} represents the expectation only on its disk of convergence. Spectral points need not satisfy $|\lambda|>\max_j\alpha_j$, so the value in \eqref{eq:spectral-renewal} need not be obtainable by summing the probability series. The polynomial identity remains valid throughout the plane. It also handles a possible eigenvalue $\lambda=0$, for which the reciprocal formulation is unavailable.

The circuit time differs from the usual first positive return time
\[
 \tau_1^+=\min\{t\geq1:X_t=1\}.
\]
If the first transition is a self-loop at stage $1$, then $\tau_1^+=1$, whereas a circuit still requires every forward edge to be traversed. Conditioning on the first transition gives
\begin{equation}\label{eq:first-return-pgf}
 F_1(z):=\mathbb E_1[z^{\tau_1^+}]
 =\alpha_1z+q_1z\prod_{j=2}^n\frac{q_jz}{1-\alpha_jz},
 \qquad |z|<R_\alpha.
\end{equation}
Here $\mathbb E_1$ denotes expectation for a chain started at state $1$. The first term describes the immediate return; in the second, the first step departs from $1$ and the remaining stages must be completed. In particular, $F_1=B_\alpha$ when $\alpha_1=0$.

This first-return formulation recovers the same determinant. With
\[
 D_{-1}(z)=\prod_{j=2}^n(1-\alpha_jz),
\]
equations~\eqref{eq:renewal-determinant} and \eqref{eq:first-return-pgf} give
\[
 \det(I-zA)=D_{-1}(z)[1-F_1(z)].
\]
For $|z|<1$, decomposing a return according to its first positive return time yields
\[
 \sum_{t=0}^\infty (A^t)_{11}z^t
 =\frac{1}{1-F_1(z)}
 =\frac{D_{-1}(z)}{\det(I-zA)}.
\]
Indeed, the return probabilities satisfy $(A^0)_{11}=1$ and
$(A^t)_{11}=\sum_{s=1}^t\mathbb P_1(\tau_1^+=s)(A^{t-s})_{11}$ for $t\geq1$; multiplying by $z^t$ and summing proves the identity. All series converge absolutely on $|z|<1$.

The holding times also identify the meaning of the parameters without reference to the spectrum:
\[
 \mathbb E T=\sum_{j=1}^n\frac1{q_j},\qquad
 \operatorname{Var}(T)=\sum_{j=1}^n\frac{\alpha_j}{q_j^2}.
\]
The chain has stationary distribution
\[
 \pi_j=\frac{q_j^{-1}}{\sum_{i=1}^nq_i^{-1}},
\]
because stationarity at stage $j$ is equivalent to
$\pi_jq_j=\pi_{j-1}q_{j-1}$. Thus mean stage durations determine the long-run occupation proportions, while \eqref{eq:spectral-renewal} relates their full geometric distributions to the spectrum.

\subsection{Continuous-time cyclic generators}

Continuous-time chains with cyclic transitions also occur in stochastic models of linear biochemical reaction chains; Mukherjee, Ghosh, and De study return times for both unidirectional and bidirectional cyclic systems~\cite{MukherjeeGhoshDe}. For the unidirectional ring, the relation to our stochastic family can be written exactly.

Let $Q(\mathbf r)$ be the generator with entries
\begin{equation}\label{eq:ring-generator}
 Q(\mathbf r)_{j,j}=-r_j,\qquad
 Q(\mathbf r)_{j,j+1}=r_j,\qquad r_j>0,
\end{equation}
and all other entries zero, with indices modulo $n$. The holding time at stage $j$ is now exponential with rate $r_j$. Choose a common rate bound $c\geq\max_jr_j$. The uniformized transition matrix is
\begin{equation}\label{eq:ring-uniformization}
 P_c=I+\frac1cQ(\mathbf r)
 =A_n\left(1-\frac{r_1}{c},\ldots,1-\frac{r_n}{c}\right).
\end{equation}
At the events of a Poisson process of rate $c$, this chain advances from stage $j$ with probability $r_j/c$ and otherwise remains there. The resulting continuous-time transition matrix is
\[
 \ee^{tQ(\mathbf r)}
 =\ee^{-ct}\sum_{k=0}^\infty\frac{(ct)^k}{k!}P_c^k,
 \qquad t\geq0,
\]
as follows by expanding $\ee^{ct(P_c-I)}$. The matrix $P_c$ describes these Poisson event opportunities; observation at a fixed time interval $t$ instead gives $\ee^{tQ(\mathbf r)}$, which can have transitions between nonadjacent stages.

\begin{corollary}[The spectral region with bounded transition rates]\label{cor:generator-region}
For every $n\geq3$ and $c>0$,
\begin{equation}\label{eq:generator-region}
 \bigcup_{\mathbf r\in(0,c]^n}\sigma\bigl(Q(\mathbf r)\bigr)
 =c(\Sigma_n-1)
 :=\{c(\lambda-1):\lambda\in\Sigma_n\}.
\end{equation}
Consequently, a prescribed $\nu\in\C$ is an eigenvalue of an admissible generator with rates at most $c$ exactly when $1+\nu/c\in\Sigma_n$. A realizing matrix $A_n(\alpha)$ for $1+\nu/c$ gives realizing rates $r_j=c(1-\alpha_j)$.
\end{corollary}

\begin{proof}
The change of variables $\alpha_j=1-r_j/c$ is a bijection between $(0,c]^n$ and $[0,1)^n$. For corresponding parameters, \eqref{eq:ring-uniformization} gives
\[
 Q(\mathbf r)=c(A_n(\alpha)-I),\qquad
 \det(\nu I-Q(\mathbf r))
 =c^n\det\bigl((1+\nu/c)I-A_n(\alpha)\bigr).
\]
The spectra are therefore related by $\nu=c(\lambda-1)$, with algebraic multiplicities preserved. Taking the union over the two parameter domains proves \eqref{eq:generator-region}; the inverse parameter change proves the realization assertion.
\end{proof}

All membership and boundary statements of \cref{thm:region,thm:boundary} transfer under this affine map. In particular, the real eigenvalue union is
\[
 \left(\bigcup_{\mathbf r\in(0,c]^n}\sigma(Q(\mathbf r))\right)\cap\R
 =\begin{cases}
  [-2c,0],&n\text{ even},\\
  (-c,0],&n\text{ odd}.
 \end{cases}
\]
The upper bound $c$ is fixed in this assertion. The union over all positive rates is obtained by taking the further union over $c>0$.

For a right eigenvector $v$, the identities $A_nv=\lambda v$ and $Qv=\nu v$ give $A_n^tv=\lambda^tv$ at integer times and $\ee^{tQ}v=\ee^{t\nu}v$ in continuous time. This explains the relevance of the spectral regions to possible oscillatory modes. A single-eigenvalue classification does not require the chosen eigenvalue to be subdominant, and the reconstruction does not prescribe the rest of the spectrum. It therefore does not by itself determine a mixing time or guarantee that the realized mode governs relaxation.

\subsection{Prescribing decay and oscillation rates}\label{sec:minimum-rates}

A prescribed eigenvalue $\nu=-d+\ii\omega$, with $d,\omega>0$, gives an eigenvector the time factor $\ee^{-dt}\ee^{\ii\omega t}$. Thus $d$ is its exponential decay rate and $\omega$ its angular frequency. The classification determines the smallest possible common upper bound on the transition rates needed to realize this eigenvalue. For the unidirectional generator $Q(\mathbf r)$ in \eqref{eq:ring-generator}, each $r_j$ is the total exit rate from state $j$, since there is only one outgoing transition.

\Needspace{27\baselineskip}
\begin{corollary}[Minimum rate bound for a prescribed eigenvalue]\label{cor:minimum-rate}
Fix $n\geq3$ and $\nu=-d+\ii\omega$ with $d,\omega>0$. Define
\[
 c_{\min}(\nu)=\inf\left\{\max_{1\leq j\leq n}r_j:
       \mathbf r\in(0,\infty)^n,\quad
       \nu\in\sigma(Q(\mathbf r))\right\},
\]
with the value $+\infty$ if the set is empty. If $\omega>d\cot(\pi/n)$, then $c_{\min}(\nu)=+\infty$. Otherwise, put
\[
 M=\Arg\nu,\qquad \mu=\mu_n(M).
\]
Then
\begin{equation}\label{eq:minimum-rate-formula}
 c_{\min}(\nu)
 =|\nu|\frac{\sin(M-\mu)}{\sin\mu},
\end{equation}
and the infimum is attained. More precisely, define
\[
 \zeta=\Lambda(\mu,M),\qquad
 \alpha_* =\frac{\zeta^n-1}{\zeta^{n-1}-1}.
\]
The denominator is nonzero, $\alpha_*\in[0,1)$, and the rates
\begin{equation}\label{eq:minimum-rate-realization}
 \mathbf r_*=
 \bigl(c_{\min}(\nu)(1-\alpha_*),
       c_{\min}(\nu),\ldots,c_{\min}(\nu)\bigr)
\end{equation}
realize $\nu$. Consequently, an admissible unidirectional ring with every exit rate at most a fixed $c>0$ realizes $\nu$ if and only if $c\geq c_{\min}(\nu)$.
\end{corollary}

\begin{proof}
For any $c>0$, \cref{cor:generator-region} makes realization with rates at most $c$ equivalent to
\[
 \lambda_c=1+\frac{\nu}{c}\in\Sigma_n.
\]
All these points lie in the upper half-plane on the same ray from $1$, with
\[
 \Arg(\lambda_c-1)=M=\Arg\nu,\qquad
 |\lambda_c-1|=\frac{|\nu|}{c}.
\]
Since $d,\omega>0$, one has $\pi/2<M<\pi$ and $\omega/d=-\tan M$. The function $-\tan M$ is strictly decreasing on this interval, and its value at $\phi_1=\pi/2+\pi/n$ is $\cot(\pi/n)$. Hence $M<\phi_1$ exactly when $\omega>d\cot(\pi/n)$. In that case \cref{thm:region} excludes every $\lambda_c$.

Suppose now that $M\geq\phi_1$. A point on this ray has angular coordinates $(m,M)$ with $0<m<M$. By \eqref{eq:ray-identity}, its distance from $1$ is
\[
 R(m,M)=\frac{\sin m}{\sin(M-m)},\qquad
 \partial_mR(m,M)=\frac{\sin M}{\sin^2(M-m)}>0.
\]
The horizontal section in \cref{prop:horizontal} is $0<m\leq\mu_n(M)$. It follows that the admissible distances along the ray form exactly the interval
\[
 0<R\leq R(\mu,M)=\frac{\sin\mu}{\sin(M-\mu)}.
\]
Its upper endpoint is included. Applying this description to $\lambda_c$ gives
\[
 \lambda_c\in\Sigma_n
 \quad\Longleftrightarrow\quad
 \frac{|\nu|}{c}\leq\frac{\sin\mu}{\sin(M-\mu)}
 \quad\Longleftrightarrow\quad
 c\geq|\nu|\frac{\sin(M-\mu)}{\sin\mu}.
\]
Both sines in the final expression are positive because $0<\mu<M<\pi$. This proves the asserted lower bound and identifies the least feasible rate cap.

To obtain the stated rates, let $k=\kappa(M)$. Then $\mu=h_k(M)\in(\gamma_k,\theta_k]$ and $\beta_k(\mu)=M$. Thus $\zeta=\Lambda(\mu,M)$ lies on the one-loop arc in \cref{prop:one-loop}. That proposition supplies a parameter $\alpha_*\in[0,1)$ for which $\zeta$ is an eigenvalue of $A_n(\alpha_*,0,\ldots,0)$. Its characteristic equation is
\[
 \zeta^{n-1}(\zeta-\alpha_*)=1-\alpha_*.
\]
If $\zeta^{n-1}=1$, this equation would give $\zeta=1$, contrary to $\Imag\zeta>0$. Solving for $\alpha_*$ therefore gives the formula in the statement. Moreover, $\zeta-1$ has direction $M$ and modulus $R(\mu,M)$, so \eqref{eq:minimum-rate-formula} gives
\[
 c_{\min}(\nu)(\zeta-1)=\nu.
\]
The generator
\[
 c_{\min}(\nu)\bigl(A_n(\alpha_*,0,\ldots,0)-I\bigr)
\]
has the rates in \eqref{eq:minimum-rate-realization} and realizes $\nu$. All its rates are positive, and their maximum is $c_{\min}(\nu)$, since $0\leq\alpha_*<1$. It is therefore admissible under every cap $c\geq c_{\min}(\nu)$.
\end{proof}

The sector restriction agrees with the classical Dmitriev--Dynkin bound~\cite{DmitrievDynkin}. The minimum in \eqref{eq:minimum-rate-formula} and the rates in \eqref{eq:minimum-rate-realization} determine the additional requirement imposed by a fixed rate cap.

At a level $M=\phi_k$ with $1\leq k\leq K$, the definition of $\kappa$ selects branch $k$, so $\mu=\theta_k$ and $\zeta=\omega_k$. The displayed optimal construction then has $\alpha_*=0$ and all rates equal to
\[
 c_{\min}(\nu)=\frac{|\nu|}{2\sin(\pi k/n)}.
\]
In particular, equality in $\omega\leq d\cot(\pi/n)$ is included and is realized by equal rates. For a lower-half-plane eigenvalue, apply the result to its conjugate; the same real generator realizes both. Real targets are governed by the real-axis statement preceding this subsection, since $M=\pi$ is outside the domain of $\mu_n$. For a target $-d<0$, the feasible caps are $c\geq d/2$ when $n$ is even and $c>d$ when $n$ is odd. The latter infimum is not attained. These endpoint distinctions concern exact eigenvalue realization; the prescribed nonreal eigenvalue need not be subdominant.

For three stages, the minimum can be expressed directly in terms of $d$ and $\omega$:
\begin{equation}\label{eq:three-stage-minimum}
 c_{\min}(-d+\ii\omega)
 =\frac{2d+\sqrt{d^2-3\omega^2}}{3},
 \qquad d>0,\quad 0<\omega\leq\frac d{\sqrt3}.
\end{equation}
An independent derivation shows how the rate constraint enters. Expanding the determinant of the three-stage generator gives
\[
 \begin{split}
 \det(sI-Q(\mathbf r))
 &=(s+r_1)(s+r_2)(s+r_3)-r_1r_2r_3\\
 &=s\bigl[s^2+(r_1+r_2+r_3)s
                    +(r_1r_2+r_2r_3+r_3r_1)\bigr].
 \end{split}
\]
Because the generator is real and has a zero eigenvalue, realizing $-d+\ii\omega$ in dimension three requires its other nonzero eigenvalue to be $-d-\ii\omega$. Comparing coefficients with
$s[s^2+2ds+(d^2+\omega^2)]$ gives the necessary and sufficient conditions
\[
 r_1+r_2+r_3=2d,\qquad
 r_1r_2+r_2r_3+r_3r_1=d^2+\omega^2.
\]
The inequality
\[
 (r_1+r_2+r_3)^2
 -3(r_1r_2+r_2r_3+r_3r_1)
 =\frac12\sum_{1\leq i<j\leq3}(r_i-r_j)^2\geq0
\]
recovers the necessary condition $d^2-3\omega^2\geq0$.

Let $c=\max_jr_j$, and choose the labels so that $r_3=c$. The sum condition gives $c\geq2d/3$ and $r_1+r_2=2d-c$. Since $r_1,r_2\leq c$,
\[
 0\leq(c-r_1)(c-r_2)
 =c^2-c(2d-c)+r_1r_2,
\]
so $r_1r_2\geq2dc-2c^2$. Substitution in the second coefficient condition yields
\[
 d^2+\omega^2
 =r_1r_2+c(2d-c)
 \geq4dc-3c^2.
\]
Thus
\[
 3c^2-4dc+d^2+\omega^2\geq0.
\]
When $0<\omega<d/\sqrt3$, the smaller root of this quadratic is strictly below $2d/3$, so the bound $c\geq2d/3$ forces
\[
 c\geq\frac{2d+\sqrt{d^2-3\omega^2}}3.
\]
When $\omega=d/\sqrt3$, the two roots equal $2d/3$ and the same conclusion follows from the sum condition.

Write $c_*=(2d+\sqrt{d^2-3\omega^2})/3$. Since $\omega>0$, one has $2d/3\leq c_*<d$. The rates
\[
 (r_1,r_2,r_3)=(2d-2c_*,c_*,c_*)
\]
are therefore strictly positive, and $2d-2c_*\leq c_*$. They have sum $2d$, while
\[
 r_1r_2+r_2r_3+r_3r_1
 =4dc_*-3c_*^2=d^2+\omega^2.
\]
They realize the required pair and attain the lower bound, proving \eqref{eq:three-stage-minimum}. At $\omega=d/\sqrt3$ all three rates are $2d/3$. As $\omega\downarrow0$, the first rate tends to zero and $c_*\to d$, which is consistent with the unattained infimum for a negative real eigenvalue in odd dimension.

Consider, in particular, the exact target
\begin{equation}\label{eq:three-stage-target}
 \nu_*=-1+\frac{\ii}{2},\qquad
 c_{\min}(\nu_*)=\frac56,\qquad
 \mathbf r_* =\left(\frac13,\frac56,\frac56\right).
\end{equation}
The corresponding characteristic polynomial is
\[
 \det(sI-Q(\mathbf r_*))
 =s\left(s^2+2s+\frac54\right),
\]
so the nonzero eigenvalues are precisely $-1\pm\ii/2$.

This example separates a localization bound from exact realization under the support restriction. Besides the sector condition, uniformization and $|\lambda|\leq1$ give the necessary disk condition
\[
 \left|1+\frac{-d+\ii\omega}{c}\right|\leq1
 \quad\Longleftrightarrow\quad
 d^2+\omega^2\leq2cd.
\]
For $d=1$, $\omega=1/2$, and $c=3/4$, both necessary bounds hold:
\[
 \frac{\omega}{d}=\frac12<\cot\frac\pi3=\frac1{\sqrt3},
 \qquad
 d^2+\omega^2=\frac54<2cd=\frac32.
\]
Nevertheless, \eqref{eq:three-stage-minimum} excludes every unidirectional three-stage realization with that cap, since $3/4<5/6$.

Allowing reverse transitions changes the minimum. Let $P_3$ be the cyclic permutation matrix and choose positive forward and backward rates
\[
 a=\frac13+\frac1{2\sqrt3},\qquad
 b=\frac13-\frac1{2\sqrt3}.
\]
The bidirectional generator is
\[
 Q_{\mathrm{bi}}=a(P_3-I)+b(P_3^{\mathsf T}-I)
 =\begin{pmatrix}
 -(a+b)&a&b\\
 b&-(a+b)&a\\
 a&b&-(a+b)
 \end{pmatrix}.
\]
Its total exit rate at every state is $a+b=2/3$. On the eigenspaces of $P_3$ corresponding to $-1/2\pm\ii\sqrt3/2$, the eigenvalues of $Q_{\mathrm{bi}}$ are
\[
 -\frac32(a+b)\pm\frac{\ii\sqrt3}{2}(a-b)
 =-1\pm\frac{\ii}{2}.
\]
Thus the same pair is realized with every total exit rate equal to $2/3$.

This rate cap is optimal among all three-state Markov generators realizing the pair, even without a support restriction. Such a generator has eigenvalues $0,-1+\ii/2,-1-\ii/2$, so its trace is $-2$. Since each diagonal entry is minus the corresponding total exit rate, the three exit rates sum to $2$; their maximum is at least $2/3$. The bidirectional construction attains this lower bound. For the pair in \eqref{eq:three-stage-target}, restricting transitions to one direction therefore raises the smallest possible maximum exit rate from $2/3$ to $5/6$. The comparison uses total exit rates in both models and concerns the effect of the prescribed support.

The target in \eqref{eq:three-stage-target} also permits a direct comparison with \cite[Theorem~1]{KolchinskyOhgaIto2026}. Their column generator is $R=Q^{\mathsf T}$, and $r=\max_i(-R_{ii})$ is the maximum total exit rate. For a one-way three-cycle the affinity is infinite and $w=1$. At cap $c=3/4$, their inequalities give
\[
 \frac{|\Imag\nu_*|}{-\Real\nu_*}
 =\frac12<\cot\frac\pi3=\frac1{\sqrt3},
 \qquad
 \frac{|\Imag\nu_*|}{2c+\Real\nu_*}
 =1<\tan\frac\pi3=\sqrt3.
\]
Replacing $r\leq c$ by $c$ preserves necessity. For the same target and cap, the left side of their Theorem~3 satisfies
\[
 \tau\frac{\sqrt3}{2}+(1-\tau)\frac1{\sqrt3}<1
 \qquad(0\leq\tau\leq1).
\]
These inequalities do not exclude $\nu_*$ at cap $3/4$, while \eqref{eq:three-stage-minimum} rules out every strictly one-way three-cycle at that cap. The cited bounds include cycles with backward transitions. The exact calculation above determines realizability under the additional support restriction and supplies the rates when realization is possible.

\subsection{Two parallel paths with equal holding parameters}

A branching support can retain the complete cycle spectrum through a quotient. Fix an integer $r\geq2$ and $p\in(0,1)$. Consider the state space
\[
 \{0,a_1,\ldots,a_r,b_1,\ldots,b_r\},
\]
where $0$ is a common hub and the two routes are
\[
 \begin{array}{l}
 0\longrightarrow a_1\longrightarrow\cdots\longrightarrow a_r\longrightarrow0,\\[3pt]
 0\longrightarrow b_1\longrightarrow\cdots\longrightarrow b_r\longrightarrow0.
 \end{array}
\]
Choose $\alpha=(\alpha_0,\ldots,\alpha_r)\in[0,1)^{r+1}$ and put $q_j=1-\alpha_j$. At the hub, the chain remains at $0$ with probability $\alpha_0$, enters the $a$-path with probability $pq_0$, and enters the $b$-path with probability $(1-p)q_0$. At the two copies of stage $j$, the holding probability is the same $\alpha_j$; an advance has probability $q_j$ and stays on the chosen path until its last stage returns to $0$.

More explicitly, the only possibly nonzero entries of the transition matrix $P_{r,p}(\alpha)$ are specified by
\begin{align*}
 P_{0,0}&=\alpha_0,&
 P_{0,a_1}&=pq_0,&
 P_{0,b_1}&=(1-p)q_0,\\
 P_{a_j,a_j}&=P_{b_j,b_j}=\alpha_j&&&(1\leq j\leq r),\\
 P_{a_j,a_{j+1}}&=P_{b_j,b_{j+1}}=q_j&&&(1\leq j<r),\\
 P_{a_r,0}&=P_{b_r,0}=q_r.
\end{align*}
Zero holding probabilities are allowed; every displayed forward edge has positive probability. Equality of the holding probabilities at corresponding stages is part of the definition of this family.

\Needspace{18\baselineskip}

\begin{proposition}[The spectrum of the parallel-path construction]\label{prop:parallel-paths}
Let $r\geq2$, $p\in(0,1)$, and $\alpha\in[0,1)^{r+1}$. Then
\begin{equation}\label{eq:parallel-factorization}
 \det\bigl(\lambda I-P_{r,p}(\alpha)\bigr)
 =\left(\prod_{j=1}^r(\lambda-\alpha_j)\right)
  \det\bigl(\lambda I-A_{r+1}(\alpha_0,\ldots,\alpha_r)\bigr).
\end{equation}
Thus the spectrum, with algebraic multiplicities, consists of the spectrum of the cyclic quotient and the additional eigenvalues $\alpha_1,\ldots,\alpha_r$. For each fixed $p\in(0,1)$,
\begin{equation}\label{eq:parallel-region}
 \bigcup_{\alpha\in[0,1)^{r+1}}\sigma(P_{r,p}(\alpha))
 =\Sigma_{r+1}\cup\{0\}
 =\cl\Sigma_{r+1}.
\end{equation}
In particular, the nonreal spectral region of this branching family is exactly the nonreal region of the $(r+1)$-cycle.
\end{proposition}

\begin{proof}
Let $P=P_{r,p}(\alpha)$ act on column vectors, regarded as complex-valued functions on the state space. A function which takes the same value at $a_j$ and $b_j$ cannot distinguish the two routes. The subspace
\[
 V=\{f:f(a_j)=f(b_j)\text{ for }1\leq j\leq r\}
\]
is invariant under $P$. Indeed, at either copy of stage $j<r$, the value of $Pf$ is $\alpha_jf(a_j)+q_jf(a_{j+1})$; at either last stage it is $\alpha_rf(a_r)+q_rf(0)$. At the hub, the two departure probabilities add to $q_0$. In coordinates $f(0),f(a_1),\ldots,f(a_r)$, the induced matrix on $V$ is therefore $A_{r+1}(\alpha_0,\ldots,\alpha_r)$.

The remaining modes record differences between the two routes. To identify an invariant complement, retain the departure weights and set
\[
 W=\{f:f(0)=0,\quad pf(a_j)+(1-p)f(b_j)=0
                         \text{ for }1\leq j\leq r\}.
\]
For $f\in W$, the hub coordinate of $Pf$ is
\[
 (Pf)(0)=q_0\bigl(pf(a_1)+(1-p)f(b_1)\bigr)=0.
\]
For $j<r$, the weighted sum of its two coordinates at stage $j$ is
\begin{align*}
 p(Pf)(a_j)+(1-p)(Pf)(b_j)
 ={}&\alpha_j\bigl(pf(a_j)+(1-p)f(b_j)\bigr)\\
   &+q_j\bigl(pf(a_{j+1})+(1-p)f(b_{j+1})\bigr)=0.
\end{align*}
At stage $r$, the same sum equals $\alpha_r$ times the corresponding sum for $f$, plus $q_rf(0)$, and also vanishes. Hence $W$ is invariant.

Since $p\in(0,1)$, a function in $W$ is determined by $f(a_1),\ldots,f(a_r)$, with $f(b_j)=-p f(a_j)/(1-p)$. In these coordinates, the restriction of $P$ to $W$ is the operator $T$ given by
\[
 (Tx)_j=\alpha_jx_j+q_jx_{j+1}\quad(1\leq j<r),
 \qquad (Tx)_r=\alpha_rx_r.
\]
The matrix of $T$ is upper triangular with diagonal $\alpha_1,\ldots,\alpha_r$, so its characteristic polynomial is $\prod_{j=1}^r(\lambda-\alpha_j)$. The dimensions of $V$ and $W$ are $r+1$ and $r$. Their intersection is zero: if $f(a_j)=f(b_j)$ and their weighted sum is zero, both values are zero, and $f(0)=0$ in $W$. Thus the whole function space is the direct sum $V\oplus W$. Taking the product of the characteristic polynomials of the two restrictions proves \eqref{eq:parallel-factorization}.

The factorization shows that every eigenvalue of the cyclic quotient occurs in the branching matrix, and that every extra eigenvalue is in $[0,1)$. Conversely, as $\alpha$ varies, the quotient ranges over the complete admissible $(r+1)$-cycle family, and each $\alpha_j$ ranges over $[0,1)$. The union on the left of \eqref{eq:parallel-region} is therefore $\Sigma_{r+1}\cup[0,1)$. By \cref{thm:region}, $(0,1]\subset\Sigma_{r+1}$, and \eqref{eq:closure} gives $\Sigma_{r+1}\cup[0,1)=\Sigma_{r+1}\cup\{0\}=\cl\Sigma_{r+1}$.
\end{proof}

The quotient is obtained by identifying $a_j$ and $b_j$ for each $j$. The proof also accounts for the eigenvalues lost under this identification: they are real because the restriction to $W$ is triangular. Consequently, the membership tests, nonreal boundary, and realizing matrices already proved for the cycle apply to this family without another boundary calculation. In particular, a cycle realization may be lifted by duplicating its non-hub stages and choosing any fixed $p\in(0,1)$. When $r+1$ is odd, zero is supplied by a difference mode with some $\alpha_j=0$, even though zero cannot occur in the admissible cyclic quotient.

\subsection{What changes when the two paths differ}\label{sec:unequal-paths}

The equal-holding hypothesis makes the quotient calculation exact. If it is removed, the support still permits a useful determinant reduction, but it no longer yields the same spectral region automatically. This can be seen before any angular analysis.

Allow two paths of lengths $r_1,r_2\geq1$, with holding probabilities $\alpha_{i,j}\in[0,1)$ on stage $j$ of path $i\in\{1,2\}$. Write $q_{i,j}=1-\alpha_{i,j}$, retain the hub holding probability $\alpha_0\in[0,1)$, and set $q_0=1-\alpha_0$. Departures from the hub choose the two paths with probabilities $p$ and $1-p$, where $p\in(0,1)$. Define
\[
 D_i(\lambda)=\prod_{j=1}^{r_i}(\lambda-\alpha_{i,j}),\qquad
 Q_i=\prod_{j=1}^{r_i}q_{i,j}.
\]
For the resulting transition matrix $P$, one has
\begin{equation}\label{eq:unequal-path-determinant}
 \begin{split}
 \det(\lambda I-P)
 ={}&(\lambda-\alpha_0)D_1(\lambda)D_2(\lambda)\\
    &-q_0\bigl[pQ_1D_2(\lambda)+(1-p)Q_2D_1(\lambda)\bigr].
 \end{split}
\end{equation}

To verify the formula, order the hub first and each path consecutively. Let $T_i$ be the upper bidiagonal transition block within path $i$, with diagonal $\alpha_{i,j}$ and superdiagonal $q_{i,j}$; the transition from its last stage to the hub is kept outside this block. For $D_i(\lambda)\ne0$, the transfer through the path is
\[
 e_1^{\mathsf T}(\lambda I-T_i)^{-1}q_{i,r_i}e_{r_i}
 =\frac{Q_i}{D_i(\lambda)}.
\]
Here $e_j$ is the appropriate coordinate vector. The equality follows by solving the triangular system backwards from its last coordinate: each stage contributes the factor $q_{i,j}/(\lambda-\alpha_{i,j})$. Taking the Schur complement of the two path blocks gives
\[
 \det(\lambda I-P)
 =D_1(\lambda)D_2(\lambda)
  \left(\lambda-\alpha_0
        -q_0p\frac{Q_1}{D_1(\lambda)}
        -q_0(1-p)\frac{Q_2}{D_2(\lambda)}\right).
\]
Clearing denominators proves \eqref{eq:unequal-path-determinant} off the finitely many zeros of $D_1D_2$. Both sides of the cleared identity are polynomials, so it holds at those zeros as well.

The probabilistic expression explains the additional determinant term. A circuit consists of the geometric holding time at the hub followed by the stages on the chosen path. Its generating function is
\[
 B(z)=\frac{q_0z}{1-\alpha_0z}
 \left[
  p\prod_{j=1}^{r_1}\frac{q_{1,j}z}{1-\alpha_{1,j}z}
  +(1-p)\prod_{j=1}^{r_2}\frac{q_{2,j}z}{1-\alpha_{2,j}z}
 \right],
\]
initially for $|z|$ below the reciprocal of the largest holding probability, with the same infinite-radius convention as before. The two possible routes produce a sum of products. When $r_1=r_2=r$ and $\alpha_{1,j}=\alpha_{2,j}=\alpha_j$, they coincide; then $D_1=D_2=D$ and $Q_1=Q_2$, and \eqref{eq:unequal-path-determinant} reduces to \eqref{eq:parallel-factorization}.

For unequal paths, the exact determinant reduction and the separate geometric holding factors remain available. Taking the argument and modulus of their sum, however, does not give the single fixed-sum constraint and separable logarithmic sum used in \cref{prop:feasibility,prop:range}. The convex range argument and the ordering of horizontal sections therefore require further justification for this larger family. Moreover, a rational equation obtained by dividing by $D_1D_2$ must be checked against the cleared polynomial at its poles: \eqref{eq:parallel-factorization} shows explicitly how eigenvalues can remain in those factors. The parallel-path result consequently applies to the stated family with equal holding parameters, while \eqref{eq:unequal-path-determinant} specifies the algebraic problem for the unrestricted branching parameters.

\section{Conclusion}\label{sec:conclusion}

The cycle support reduces the eigenvalue equation to one product identity. In angular coordinates, that identity becomes a fixed-sum constraint together with a strictly convex logarithmic modulus sum. The equal-angle vector and the simplex vertices give the two extremal matrix families. Their branch graphs are decreasing, and the separation $\theta_k<\gamma_{k+1}$ orders all active horizontal sections. The single endpoint function $\mu_n$ then yields the complete region, its boundary, and realizing matrices without an additional global spectral theorem.

The algebraic equation is useful for evaluation, but the angular inequalities supply the component selection and parameter admissibility that the polynomial equation alone does not contain. The same distinction is important at the real endpoints: a limiting one-loop parameter identifies $0$ as a boundary point, whereas its actual attainment is a separate parity question.

For continuous-time one-way cycles, the same boundary answers an extremal question about the rates. Fixing a nonreal eigenvalue fixes a ray from $1$ under uniformization; its furthest admissible point determines the minimum rate bound in \cref{cor:minimum-rate}. The boundary realization attains that bound with at most one slower stage. The three-stage comparison in \cref{sec:minimum-rates} shows the effect of the support restriction explicitly: the same decay and oscillation rates require different minimum exit-rate bounds when backward transitions are allowed.

The parallel-path construction in \cref{prop:parallel-paths} applies the classification to a branching support through invariant subspaces. With matching holding probabilities, the additional modes are real and the cyclic quotient supplies the entire nonreal spectrum. The comparison with unequal paths in \cref{sec:unequal-paths} shows why a sparse determinant alone does not yield the same convex range problem: several return paths introduce a sum of products. The product reduction, complete range, and section ordering are therefore separate properties to establish when adapting the argument to another support.

\section*{Acknowledgements}
Brecht Verbeken thanks his colleague and friend Brando Vagenende, first author of the earlier $4$-cycle paper~\cite{VagenendeEtAl}. That work is the motivating special case and conceptual precursor of the present classification. The all-dimensional argument was developed independently, but is heavily indebted to the earlier result, which supplied the starting point for the general geometry.

\paragraph{Funding.}
Vincent Ginis acknowledges support from the Research Foundation--Flanders (FWO) under grants No.~G032822N and G0K9322N.

\paragraph{Competing interests.}
The authors declare that they have no competing interests.

\paragraph{Data availability.}
This work is theoretical and uses no empirical data. The figures illustrate the geometric sets defined in the statements.

\paragraph{Declaration of generative AI and AI-assisted technologies.}
ChatGPT 5.5 (OpenAI) was used to assist with proof organization and exposition, LaTeX preparation, and reproducible plotting code. ChatGPT 6 Astra (OpenAI) was used for rewriting. The authors reviewed the mathematical arguments and take responsibility for the content of the manuscript.

\end{document}